\documentclass{amsart}
\usepackage{amsmath}
\usepackage{amsfonts}
\usepackage{amssymb}
\usepackage{mathrsfs}
\usepackage{amscd}
\usepackage{url}

\newcommand{\A}{\mathcal{A}}

\newcommand{\re}{\mathbb{R}}
\newcommand{\cpx}{\mathbb{C}}

\newcommand{\N}{\mathbb{N}}

\newcommand{\lmd}{\lambda}

\newcommand{\eps}{\epsilon}

\newcommand{\dt}{\delta}
\newcommand{\Dt}{\Delta}
\def\af{\alpha}
\def\bt{\beta}

\def\rank{\mbox{rank}}

\newcommand{\Sig}{\Sigma}

\newcommand{\reff}[1]{(\ref{#1})}

\newcommand{\prm}{\prime}

\newcommand{\supp}[1]{\mbox{supp}(#1)}

\newcommand{\bdes}{\begin{description}}
\newcommand{\edes}{\end{description}}

\newcommand{\bal}{\begin{align}}
\newcommand{\eal}{\end{align}}

\newcommand{\bnum}{\begin{enumerate}}
\newcommand{\enum}{\end{enumerate}}

\newcommand{\bit}{\begin{itemize}}
\newcommand{\eit}{\end{itemize}}

\newcommand{\bea}{\begin{eqnarray}}
\newcommand{\eea}{\end{eqnarray}}
\newcommand{\be}{\begin{equation}}
\newcommand{\ee}{\end{equation}}

\newcommand{\baray}{\begin{array}}
\newcommand{\earay}{\end{array}}

\newcommand{\bsry}{\begin{subarray}}
\newcommand{\esry}{\end{subarray}}

\newcommand{\bca}{\begin{cases}}
\newcommand{\eca}{\end{cases}}

\newcommand{\bcen}{\begin{center}}
\newcommand{\ecen}{\end{center}}

\newcommand{\bbm}{\begin{bmatrix}}
\newcommand{\ebm}{\end{bmatrix}}

\newcommand{\bmx}{\begin{matrix}}
\newcommand{\emx}{\end{matrix}}

\newcommand{\bpm}{\begin{pmatrix}}
\newcommand{\epm}{\end{pmatrix}}

\newcommand{\btab}{\begin{tabular}}
\newcommand{\etab}{\end{tabular}}

\newtheorem{theorem}{Theorem}[section]

\newtheorem{prop}[theorem]{Proposition}
\newtheorem{lem}[theorem]{Lemma}

\newtheorem{ass}[theorem]{Assumption}

\theoremstyle{definition}

\newtheorem{exm}[theorem]{Example}
\newtheorem{alg}[theorem]{Algorithm}

\newtheorem{remark}[theorem]{Remark}

\begin{document}

%
%
%
%
%
%
%
%
%

\title[Linear Optimization with Moments and Nonnegative Polynomials]
{Linear Optimization with Cones of
Moments and Nonnegative Polynomials}

\author[Jiawang Nie]{Jiawang Nie}
\address{Department of Mathematics\\
  University of California \\
  San Diego}
\email{njw@math.ucsd.edu}
\thanks{Research was partially supported by the 
NSF grants DMS-0844775 and DMS- 1417985.}

\begin{abstract}
Let $\A$ be a finite subset of $\N^n$ and $\re[x]_{\A}$ be the
space spanned by monomials $x^\af$ with $\af \in \A$.
Let $K$ be a compact semialgebraic set of $\re^n$ such that
a polynomial in $\re[x]_{\A}$ is positive on $K$.
Denote by $\mathscr{P}_{\A}(K)$ the cone of polynomials in $\re[x]_{\A}$
that are nonnegative on $K$. The dual cone of
$\mathscr{P}_{\A}(K)$ is $\mathscr{R}_{\A}(K)$,
the set of all moment sequences in $\re^{\A}$
that admit representing measures supported in $K$.
First, we study geometric properties of
$\mathscr{P}_{\A}(K)$ and $\mathscr{R}_{\A}(K)$
(like interiors, closeness, duality, memberships),
and construct a convergent hierarchy of semidefinite relaxations for each of them.
Second, we propose a semidefinite algorithm for solving
linear optimization problems with the cones
$\mathscr{P}_{\A}(K)$ and $\mathscr{R}_{\A}(K)$,
and prove its asymptotic and finite convergence.
Third, we show how to check whether $\mathscr{P}_{\A}(K)$ and $\mathscr{R}_{\A}(K)$
intersect affine subspaces;
if they do, we show how to get a point in the intersections;
if they do not, we prove certificates for the non-intersecting.
\end{abstract}



\maketitle

\section{Introduction}

Let $\N$ (resp., $\re$) be the set of nonnegative integers
(resp., real numbers), and
let $\re[x]:=\re[x_1,\ldots,x_n]$ be the ring of real polynomials in
$x:=(x_1,\ldots,x_n)$.
For $\af :=(\af_1, \ldots,\af_n) \in \N^n$,
denote $x^\af := x_1^{\af_1}\cdots x_n^{\af_n}$ and $|\af|:=\af_1+\cdots+\af_n$.
Let $\re[x]_d$ be the set of polynomials in $\re[x]$ with degrees $ \leq d$,
and $K \subseteq \re^n$ be a set. Denote
\[
\mathscr{P}_d(K) =\{ p \in \re[x]_d: \, p(u) \geq 0 \,\,\forall \, u \in K \},
\]
the cone of polynomials in $\re[x]_d$ that are nonnegative on $K$.
Let
\[ \N_d^n=\{\af \in \N^n: |\af| \leq d\}. \]
The dual space of $\re[x]_d$ is $\re^{\N_d^n}$,
the space of {\it truncated moment sequences} (tms') of degree $d$.
A tms $y = (y_\af) \in \re^{\N_d^n}$
defines a linear functional acting on $\re[x]_d$ as
{\small
\[
\langle p, y \rangle := \sum_{|\af|\leq d}
p_\af y_\af \quad \mbox{ for all } \quad
p = \sum_{|\af|\leq d} p_\af x^\af.
\]
\noindent}It is said to admit a $K$-measure $\mu$
(i.e., $\mu$ is a Borel measure supported in $K$) if
$y_\af = \int x^\af \mathtt{d}\mu$ for all $\af \in \N_d^n$.
Such $\mu$ is called a {\it $K$-representing} measure for $y$.
In applications, we are often interested in {\it finitely atomic} measures,
i.e., their supports are finite sets.
Denote by $\dt_u$ the Dirac measure supported at $u$.
A measure $\mu$ is called {\it $r$-atomic} if
$\mu = \lmd_1 \dt_{u_1} + \cdots + \lmd_r \dt_{u_r}$
with each $\lmd_i >0$ and $u_i \in \re^n$.
Let $meas(y,K)$ be the set of all $K$-measures admitted by $y$.
Denote
\[
\mathscr{R}_d(K) = \{ y \in \re^{\N_d^n} : \, meas(y,K) \ne \emptyset \}.
\]
When $K$ is compact, $\mathscr{R}_d(K)$ is the dual cone of $\mathscr{P}_d(K)$
(cf. Tchakaloff \cite{Tch} and Laurent \cite[Section~5.2]{Lau}).

Linear optimization problems with cones $\mathscr{P}_d(K)$ and $\mathscr{R}_d(K)$
have wide applications. For instance, the minimum value
of a polynomial $f \in \re[x]_d$ on $K$ can be found by maximizing $\gamma$
subject to $f-\gamma \in \mathscr{P}_d(K)$; the corresponding dual problem is
minimizing a linear function over the cone $\mathscr{R}_d(K)$
(cf. Lasserre \cite{Las01}). Generalized problems of moments (GPMs),
proposed by Lasserre \cite{Las08},
are optimizing linear moment functionals
over the set of measures supported in a given set
and satisfying some linear constraints.
GPMs are equivalent to linear optimization problems with the cone $\mathscr{R}_d(K)$.
Lasserre \cite{Las08} proposed semidefinite relaxations for solving GPMs.
We refer to \cite{GLPT12,LasBok,Lau07,Lau,ParStu,Par03}
for moment and polynomial optimization problems.
Semidefinite programs are also very useful in representing convex sets and convex hulls,
like in \cite{GN11,GPT12,HN1,HN2,Las08s,Las09,NPS10,Schd12},
and  in solving polynomial equations,
like in \cite{LLR08,LLR09,LR12}.
%
%
%

\bigskip \noindent
{\bf Motivations} \,
In some applications and mathematical problems, we do not have all
entries of a truncated moment sequence,
or we only require partial entries of a tms
to satisfy certain properties. For instance, does there exist a tms $y$
that admits a measure supported in the circle $x_1^2+x_2^2=1$ and
satisfy the linear equations
\be  \label{eqn:y224060=0}
y_{22} = y_{40}+y_{04} = y_{60}+y_{06} =1\,?
\ee
This is a moment problem, but it only involves five moments
$y_{22}, y_{40}, y_{04}, y_{60}, y_{06}$.
As we will see in Example~\ref{exm:5.4}, such a tms $y$ does not exist.

The above motivates us to consider more general settings
of moment problems. Sometimes, we need to work
in the space of incomplete truncated moment sequences.
This leads to the $\A$-truncated $K$-moment problem ($\A$-TKMP),
proposed and studied in \cite{Nie-ATKMP}.
Let $\A \subseteq \N_d^n$ be a subset. The space $\re^{\A}$
is the set of all partially truncated moment sequences,
which only have moments indexed by $\af \in \A$.
An element in $\re^{\A}$ is called an $\A$-truncated moment sequence ($\A$-tms).
A basic question in $\A$-TKMP is:
does an $\A$-tms have a $K$-representing measure?
This issue was discussed in \cite{Nie-ATKMP}.
A generalization of this question is the moment completion problem (MCP):
given an $\A$-tms $y$, can we extend it to a tms $z \in  \mathscr{R}_d(K)$
such that it satisfies some properties?

For the above observations, we consider generalizations of the cones
$\mathscr{R}_d(K)$ and $\mathscr{P}_d(K)$.
Let $\A$ be a finite set in $\N^n$, and
$\re[x]_{\A} := \mbox{span}\{x^\af: \, \af \in \A\}$.
The dual space of $\re[x]_{\A}$ is $\re^{\A}$.
Define $\deg(\A) :=\max\{ |\af|: \, \af \in \A\}$.
The $K$-representing measures for an $\A$-tms $y$ and the set $meas(y,K)$
can be defined same as before. Denote
\[
\baray{rcl}
\mathscr{P}_{\A}(K) &=& \{ p \in \re[x]_{\A}:  p(u) \geq 0 \, \forall \, u \in K \}, \\
\mathscr{R}_{\A}(K) &=& \{ y \in \re^{\A}: \, meas(y,K) \ne \emptyset \}.
\earay
\]
Clearly, if $\A=\N_d^n$, then $\mathscr{P}_{\A}(K)=\mathscr{P}_d(K)$
and $\mathscr{R}_{\A}(K) = \mathscr{R}_d(K)$.
An $\A$-tms $y \in  \mathscr{R}_d(K)$
if and only if it admits a $r$-atomic $K$-measure with $r\leq |\A|$
(cf. \cite[Proposition~3.3]{Nie-ATKMP}).

The $\A$-TKMP has applications in solving moment problems
with noncompact sets like $\re^n$. A classical moment problem is checking
the membership in $\mathscr{R}_d(\re^n)$.
This question is harder than other moment problems because $\re^n$ is noncompact.
However, it can be transformed to a
compact moment problem via homogenization (cf.~\cite{FiNi2012}).
Note that a tms in $\re^{\N_d^n}$ can be
thought of as an $\A$-tms in $\re^{\N_d^{n+1}}$
with $\A = \{ \bt \in \N^{n+1}: |\bt| = d\}$.
Then, under some general assumptions,
$y\in \mathscr{R}_d(\re^n)$ if and only if $y$,
as an $\A$-tms in $\re^{\A}$, belongs to $\mathscr{R}_{\A}(\mathbb{S}^n)$ where
$
\mathbb{S}^n = \{ \tilde{x} \in \re^{n+1}: \, \| \tilde{x} \|_2 =1 \}
$
is the $n$-dimensional unit sphere.
We refer to \cite{FiNi2012} for more details.
The latter question is an $\A$-TKMP with the compact set $\mathbb{S}^n$, and
it can be solved by the method in \cite{Nie-ATKMP}.

The $\A$-TKMP has applications in sums of even power (SOEP) of linear forms.
A form (i.e., a homogeneous polynomial) $f \in \re[x]_d$
($d$ is even) is said to be SOEP
if there exist $L_1,\ldots, L_r \in \re[x]_1$ such that
$f = L_1^d+\cdots + L_r^d$ (cf. \cite{Rez92}).
Let $Q_{n,d}$ be the cone of all SOEP forms of degree $d$.
Each $f$ can be written as
{\small \small
\[
f = \sum_{ \substack{ \af = (\af_1,\ldots, \af_n) \in \N^n \\  |\af| =d  }  }
\binom{d}{\af_1,\ldots,\af_n}
\check{f}_\af x^\af.
\]
\noindent}So,
$f$ can be identified as an $\A$-tms $\check{f} \in \re^\A$
with $\A=\{\af \in \N^n: |\af|=d\}$. Indeed, $f \in Q_{n,d}$ if and only if
$\check{f} \in \mathscr{R}_{\A}(\mathbb{S}^{n-1})$
(cf.~\cite{Nie-ATKMP,Rez92}). Its dual cone $\mathscr{P}_{\A}(K)$ is
$P_{n,d}$, the cone of nonnegative forms in $n$ variables and of degree $d$.
So, checking SOEP forms is an $\A$-TKMP
over the compact set $\mathbb{S}^{n-1}$.

Another application of $\A$-TKMP  is in completely positive (CP) matrices.
A symmetric matrix $C \in \re^{n\times n}$ is CP if
$C = u_1 u_1^T + \cdots + u_ru_r^T$
for $u_1,\ldots, u_r \in \re_+^n$ (the nonnegative orthant).
Each symmetric matrix can be thought of an $\A$-tms in
$\re^{\A}$ with $\A=\{\af \in \N^n: |\af|=2\}$.
It can be shown that $C$ is CP if and only if
$C$, as an $\A$-tms, belongs to $\mathscr{R}_{\A}(K)$
with $K=\{x\in\re_+^n:\, x_1+\cdots+x_n=1\}$ (cf.~\cite{Nie-ATKMP}).
This is also an $\A$-TKMP over a compact set.
The dual cone is $\mathscr{P}_{\A}(K)$, the cone
of $n\times n$ copositive matrices.
(A symmetric matrix $B$ is copositive if $x^TBx \geq 0$ for all $x \geq 0$.)
We refer to \cite{BerSM03,Dur10} for copositive and CP matrices.
An important question is the CP-completion problem is:
given a partial symmetric matrix $A$
(i.e., only its partial entries are known),
we want to assign values to its unknown entries so that $A$ is CP.
This problem is recently investigated by Zhou and Fan \cite{ZhouFan}.
They formulated the CP-completion problem
as an $\A$-TKMP.

\bigskip \noindent
{\bf Contributions} \,
Assume $\A \subseteq \N^n$ is finite and $K$ is
a semialgebraic set as
\be  \label{def:K}
K = \left\{x\in \re^n :\, h(x)  = 0,  g(x) \geq 0 \right\},
\ee
defined by two polynomial tuples
$h=(h_1,\ldots, h_{m_1})$ and $g=(g_1,\ldots, g_{m_2})$.
Assume $K$ is compact and $\re[x]_{\A}$ contains a polynomial
that is positive on $K$. In the recent work \cite{Nie-ATKMP}
by the author, a method is given on how to check whether or not
a given $\A$-tms $y\in \re^{\A}$ belongs to the cone $\mathscr{R}_{\A}(K)$.
In particular, by the method in \cite{Nie-ATKMP},
we can check whether or not a given form $f \in \re[x]_d$ ($d$ is even)
is a sum of even powers of real linear forms,
and we can check whether or not
a given symmetric matrix is completely positive.
In \cite{Nie-ATKMP}, the $\A$-tms $y$ is assumed to be known, i.e.,
all the entries $y_\af$ ($\af \in \A$) are given.
However, in some applications, we often do not know all the moment values $y_\af$,
but only know they satisfy some linear equations.
For instance, how do we know whether there exists an $\A$-tms,
which admits a measure supported in the circle $x_1^2+x_2^2=1$
and satisfies the linear equations \reff{eqn:y224060=0}?
In such occasions, we often need to know whether or not
there exists $y \in \mathscr{R}_{\A}(K)$ satisfying the given equations.
The method in \cite{Nie-ATKMP} cannot solve such questions.

Considering the above, we study more general
linear optimization problems with the cone $\mathscr{R}_{\A}(K)$.
That is, we discuss how to minimize a linear objective function in $y \in \re^{\A}$,
subject to linear equations in $y$ and the membership constraint
$y \in \mathscr{R}_{\A}(K)$. If the objective does not depend on $y$
(i.e., it is a constant), then the problem is reduced to checking
whether exists $y \in \mathscr{R}_{\A}(K)$ satisfying
a set of linear equations. This is a feasibility question.
Linear optimization problems with
the cone $\mathscr{P}_{\A}(K)$ will also be studied.

First, we study properties of the cones
$\mathscr{P}_\A(K)$ and $\mathscr{R}_\A(K)$.
We characterize their interiors, prove their closeness and dual relationship,
i.e., $\mathscr{R}_\A(K)$ is the dual cone of $\mathscr{P}_\A(K)$.
We construct a convergent hierarchy of semidefnite relaxations for each of them.
%
%
This will be shown in Section~\ref{sec:intclose}.

Second, we study how to solve linear optimization problems
with cones $\mathscr{P}_\A(K)$ and $\mathscr{R}_\A(K)$.
A semidefinite algorithm is proposed for solving them.
Its asymptotic and finite convergence are proved.
A stopping criterion is also given.
This will be shown in Section~\ref{sec:liopt}.

Third, we study how to check whether an affine subspace intersects
the cone $\mathscr{P}_\A(K)$ or $\mathscr{R}_\A(K)$.
If they intersect, we show how to find a point in the intersection.
If they do not, we prove certificates for the non-intersecting.
This will be shown in Section~\ref{sec:feas}.

We begin with a review of some basics in the field
in Section~\ref{sec:prlim}.

\section{Preliminaries}  \label{sec:prlim}
\setcounter{equation}{0}

\noindent
{\bf Notation} \,
For $t\in \re$, $\lceil t\rceil$ (resp., $\lfloor t\rfloor$)
denotes the smallest integer not smaller
(resp., the largest integer not greater) than $t$.
For $k \in \N$, denote $[k]:=\{1,\ldots,k\}$.
For a tms $z$, denote by $z|_{\A}$ the subvector of $z$
whose indices are in $\A$.
When $\A=\N_d^n$, we simply denote $z|_d := z|_{\N_d^n}$.
For a set $S \subseteq \re^n$, $|S|$ denotes its cardinality,
and $int(S)$ denotes its interior.
The superscript $^T$ denotes the transpose of a matrix or vector.
For $u\in \re^N$ and $r \geq 0$, denote $\| u \|_2 := \sqrt{u^Tu}$
and $B(u,r) :=\{x\in \re^n \mid \|x-u\|_2 \leq r\}$.
For a polynomial $p\in \re[x]$, $\|p\|_2$ denotes the $2$-norm
of the coefficient vector of $p$.
For a matrix $A$, $\|A\|_F$ denotes its Frobenius norm.
If a symmetric matrix $X$ is positive semidefinite (resp., definite),
we write $X\succeq 0$ (resp., $X\succ 0$).

\subsection{Riesz functionals, localizing matrices and flatness}

Let $\A \subseteq \N^n$.
An $\A$-tms $y$ defines a Riesz functional $\mathscr{L}_y$ acting on $\re[x]_{\A}$ as
{\small
\[
\mathscr{L}_y \Big( \sum_{\af\in\A} p_\af x^\af \Big ) :=
\sum_{\af\in\A} p_\af  y_\af.
\]
\noindent}Denote
$\langle p, y \rangle := \mathscr{L}_y(p)$ for convenience.
We say that $\mathscr{L}_y$ is {\it $K$-positive} if
\[
\mathscr{L}_y(p) \geq 0 \quad \forall \,
p \in  \mathscr{P}_{\A}(K),
\]
and $\mathscr{L}_y$ is {\it strictly} $K$-positive if
\[
\mathscr{L}_y(p) > 0 \quad \forall \,
p \in  \mathscr{P}_{\A}(K): \, p|_K \not\equiv 0.
\]
As is well known, $\mathscr{L}_y$ being $K$-positive is a necessary condition
for $y$ to admit a $K$-measure. The reverse is also true
if $K$ is compact and $\re[x]_{\A}$ is {\it $K$-full}
(i.e., there exists $p\in \re[x]_{\A}$ such that $p>0$ on $K$)
(cf. \cite[Theorem~2.2]{FiNi2012}).
We refer to the appendix for how to check whether $\re[x]_{\A}$ is $K$-full or not.

For $q\in \re[x]_{2k}$,
define $L_{q}^{(k)}(z)$ to be the symmetric matrix such that
\be \label{df:Loc-Mat}
\mathscr{L}_z(qp^2) \, = \, p^T
\left( L_{q}^{(k)}(z) \right)   p \quad \forall
p \in \re[x]: \, \deg(qp^2) \leq 2k.
\ee
(For convenience, we still use $p$ to denote the vector of coefficients of
a polynomial $p$, indexed by monomial powers $\af \in \N^n$.)
The matrix $L_{q}^{(k)}(z)$ is called the  $k$-th order
{\it localizing matrix} of $q$ generated by $z$.
When $q=1$, the matrix
\[
M_k(z) := L_{q}^{(k)}(z)
\]
is called the $k$-th order {\it moment matrix} of $z$.
The rows and columns of $L_{q}^{(k)}(z)$ are indexed by $\af \in \N^n$.
We refer to \cite{LasBok,Lau} for moment and localizing matrices.

Let $K$ be as in \reff{def:K} and $g_0=1$.
A necessary condition for $z \in \mathscr{R}_{2k}(K)$ is
\be  \label{loc-M>=0}
L_{h_i}^{(k)}(z) = 0 \, ( i =1,\ldots, m_1), \quad
L_{g_j}^{(k)}(z) \succeq 0  \, ( j =0,1,\ldots, m_2).
\ee
(Cf.~\cite{CF05,Nie-ATKMP}.)
Define the integer $d_K$ as ($h_i$, $g_j$ are from \reff{def:K})
\be \label{df:dg}
d_K  := \max_{i \in [m_1],  j \in [m_2]}
\{1, \lceil \deg(h_i)/2 \rceil, \lceil \deg(g_j)/2 \rceil\} .
\ee
In addition to \reff{loc-M>=0}, if $z$ also satisfies the rank condition
\be \label{cond:FEC}
\rank \, M_{k-d_K}(z) = \rank \, M_k(z),
\ee
then $z$ admits a unique $K$-measure, which is finitely atomic
(cf.~Curto and Fialkow \cite{CF05}).
For convenience, throughout the paper, we simply say $z$ is {\it flat}
if \reff{loc-M>=0} and \reff{cond:FEC} hold for $z$.
For a flat tms, its finitely atomic representing measure can be found
by solving some eigenvalue problems (cf.~Henrion and Lasserre~\cite{HenLas05}).
Flatness is very useful for solving truncated moment problems,
as shown by Curto and Fialkow \cite{CF96,CF98,CF05}.
A nice exposition for flatness can also be found in Laurent \cite{Lau05}.

For $z\in \re^{\N_{2k}^n}$ and $y\in \re^{\A}$, if $z|_{\A} = y$,
we say that $z$ is an {\it extension} of $y$,
or equivalently, $y$ is a {\it truncation} of $z$.
Clearly, if $z$ is flat and $y=z|_\A$, then $y$ admits a $K$-measure.
In such case, we say $z$ is a flat extension of $y$.
Thus, the existence of a $K$-representing measure for $y$
can be determined by investigating whether
$y$ has a flat extension or not.
This approach has been exploited in \cite{HN4,Nie-ATKMP}.

\subsection{Ideals, quadratic modules and positive polynomials}
\label{sec:2.2}

A subset $I \subseteq \re[x]$ is called an {\it ideal} if
$I + I \subseteq I$ and $ I \cdot \re[x] \subseteq I$.
For a tuple $p=(p_1,\ldots,p_m)$ of polynomials in $\re[x]$,
denote by $I(p)$ the ideal generated by $p_1,\ldots, p_m$,
which is the set
$p_1 \re[x] + \cdots + p_m \re[x]$.
A polynomial $f$ is called a {\it sum of squares} (SOS)
if there exist $f_1,\ldots,f_k \in \re[x]$ such that
$f=f_1^2+\cdots+f_k^2$.
The cone of all SOS polynomials in $n$ variables and of degree $d$
is denoted by $\Sig_{n,d}$.
We refer to Reznick~\cite{Rez00} for a survey on SOS polynomials.

Let $h=(h_1,\ldots,h_{m_1})$ and $g=(g_1,\ldots,g_{m_2})$
be as in \reff{def:K}. Denote
%
%
\be \label{t-ideal:h}
I_{\ell}(h)  = h_1 \re[x]_{\ell-\deg(h_1)}
+ \cdots + h_{m_1} \re[x]_{\ell-\deg(h_{m_1})},
\ee
%
%
\be \label{qmod-g}
Q_k(g)  =  \Sig_{n,2k} + g_1 \Sig_{n,2k-\deg(g_1)} +
\cdots + g_{m_2} \Sig_{n,2k-\deg(g_{m_2})}.
\ee
Clearly, $I(h) = \cup_{k\in \N} I_{2k}(h)$.
The union $Q(g):= \cup_{k\in \N} Q_k(g)$
is called the {\it quadratic module} generated by $g$.
Clearly, if $f\in I(h) +  Q(g)$, then $f|_K \geq 0$.
The converse is also true if $p|_K >0$ and $I(h)+Q(g)$ is archimedean
(i.e., there exists $R>0$ such that $R-\|x\|_2^2 \in I(h) + Q(g)$).
This is called Putinar's Positivstellensatz.

\begin{theorem}[Putinar, \cite{Put}]\label{thm:PutThm}
Let $K$ be as in \reff{def:K}. Suppose $I(h)+Q(g)$ is archimedean.
If $f\in \re[x]$ is positive on $K$, then $f\in I(h) + Q(g)$.
\end{theorem}

Let $h$ and $g$ be as in \reff{def:K}. Denote
{\small
\be \label{df:Phi(g)}
\Phi_k(g) := \left\{
\left. w \in \re^{ \N_{2k}^n } \right|
L_{g_j}^{(k)}(w) \succeq 0, \,\,\,
j=0,1,\ldots, m_2 \right\},
\ee
\be \label{df:Ek(h)}
E_k(h) := \left\{
\left. w \in \re^{ \N_{2k}^n } \right|
L_{h_i}^{(k)}(w) = 0, \ \,\, \,
i=1,\ldots,m_1 \right\}.
\ee
}
The set $I_{2k}(h)+Q_k(g)$ is dual to $\Phi_k(g) \cap E_k(h)$
(cf. \cite{LasBok,Lau,Nie-ATKMP}), i.e.,
\be \label{<p,z>=0:qmod}
\langle p, z \rangle \geq 0 \qquad
\forall \, p \in I_{2k}(h)+Q_k(g), \, \,
\forall \, z \in \Phi_k(g) \cap E_k(h).
\ee

\section{Properties of $\mathscr{R}_{\A}(K)$ and $\mathscr{P}_{\A}(K)$}
\label{sec:intclose}
\setcounter{equation}{0}

This section studies geometric properties
of the cones $\mathscr{R}_{\A}(K)$ and $\mathscr{P}_{\A}(K)$.

\subsection{Interiors, closedness and duality}

Recall that $\re[x]_{\A}$ is {\it $K$-full} if there exists
$p \in \re[x]_{\A}$ such that $p>0$ on $K$.
%
%
The interiors of $\mathscr{R}_{\A}(K)$ and $\mathscr{P}_{\A}(K)$
can be characterized as follows.

\begin{lem} \label{lm:int-ARz}
Let $K \subseteq \re^n$ be a nonempty compact set.
Suppose $\A \subseteq \N^n$ is finite and $\re[x]_{\A}$ is $K$-full.
Then we have:
\bit

\item [(i)] A polynomial $f \in \re[x]_{\A}$ lies in the interior of $\mathscr{P}_{\A}(K)$
if and only if $f>0$ on $K$.

\item [(ii)] An $\A$-tms $y \in \re^{\A}$ lies in the interior of $\mathscr{R}_{\A}(K)$
if and only if the Riesz functional $\mathscr{L}_y$ is strictly $K$-positive.

\eit
\end{lem}
\begin{proof}
(i) If $f>0$ on $K$, then $f \in \mbox{int}\big(\mathscr{P}_{\A}(K) \big)$.
This is because $f + q > 0$ on $K$
for all $q \in \re[x]_\A$ with sufficiently small coefficients
(the set $K$ is compact).
Conversely, suppose $f \in \mbox{int}\big(\mathscr{P}_{\A}(K) \big)$.
Since $\re[x]_\A$ is $K$-full, there exists $p \in \re[x]_\A$ with $p>0$ on $K$.
Then $f - \eps p \in \mathscr{P}_{\A}(K)$ for some $\eps>0$. So,
$f \geq \eps p > 0$ on $K$.

(ii) Let $\tau$ be a probability measure on $\re^n$ whose support equals $K$.
(Because $K$ is nonempty and compact, such a measure always exists,
as shown in Rogers~\cite{Rog85}.)
For all $p \in \mathscr{P}_{\A}(K)$, $p|_K \not\equiv 0$
if and only if $\int p \mathtt{d}\tau >0$. Let
{\small
\[
z = \int [x]_{\A} \mathtt{d} \tau \in \re^\A, \quad
\mathscr{P}_{\A}(K,\tau) = \left \{ p \in \mathscr{P}_{\A}(K):\,
 \int p \mathtt{d} \tau = 1  \right \}.
\]
} \noindent
(The $[x]_{\A}$ is the vector of monomials $x^\af$
with $\af \in \A$.) Note that
an $\A$-tms $w$ is $K$-positive (resp., strictly $K$-positive)
if and only if $\mathscr{L}_w(p) \geq 0$ (resp., $>0$)
for all $p \in \mathscr{P}_{\A}(K,\tau)$.
So, $z$ is strictly $K$-positive.

\noindent
$``\Rightarrow"$
Suppose $y \in  int (\mathscr{R}_{\A}(K))$. Then
$w:=y-\eps z \in \mathscr{R}_{\A}(K)$ for some $\eps>0$ and
\[
\mathscr{L}_y(p) = \mathscr{L}_w(p) + \eps \mathscr{L}_{z}(p)
\geq \eps \mathscr{L}_{z}(p) > 0
\]
for all $p \in \mathscr{P}_{\A}(K,\tau)$.
So, $\mathscr{L}_y$ is strictly $K$-positive.

\smallskip
\noindent
$``\Leftarrow"$ Suppose $\mathscr{L}_y$ is strictly $K$-positive.
The set $\mathscr{P}_{\A}(K,\tau)$ is compact. Let
\[
\eps =  \min\left \{\mathscr{L}_y(p) :\,
p \in \mathscr{P}_{\A}(K,\tau)  \right \} > 0,
\]
\[
M  =  \max \left \{ |\langle z, p \rangle| :\,
z \in \re^{\A}, \|z\|_2 =1,  p \in \mathscr{P}_{\A}(K,\tau)  \right \}.
\]
For all $w\in \re^{\A}$ with $\|w-y\|_2 < \frac{\eps}{2M}$,
it holds that for all $p \in \mathscr{P}_{\A}(K,\tau)$,
\[
\mathscr{L}_w(p) = \mathscr{L}_y(p) + \mathscr{L}_{w-y}(p)
\geq (\eps - \|w-y\|_2 M ) > 0.
\]
This means that all such $w$ are $K$-positive.
Because $\re[x]_{\A}$ is $K$-full, by Theorem~2.2 of \cite{FiNi2012},
every such $w$ belongs to $\mathscr{R}_{\A}(K)$.
So, $y$ is an interior point of $\mathscr{R}_{\A}(K)$.
\qed
\end{proof}

The dual cone of $\mathscr{P}_{\A}(K)$ is defined as
\[
\mathscr{P}_{\A}(K)^* := \{ y \in \re^{\A}: \,
\langle p, y\rangle \geq 0 \,\, \forall \, p \in \mathscr{P}_{\A}(K) \}.
\]
When $\A = \N_d^n$ and $K$ is compact,
$\mathscr{P}_d(K)^*=\mathscr{R}_d(K)$
(cf. Tchakaloff~\cite{Tch}, Laurent~\cite[Section~5.2]{Lau}).
For more general $\A$, a similar result holds.

\begin{prop} \label{lm:pro+dual}
Let $K \subseteq \re^n$ be a nonempty compact set.
Suppose $\A \subseteq \N^n$ is finite and $\re[x]_{\A}$ is $K$-full.
Then, the cones $\mathscr{R}_{\A}(K)$ and $\mathscr{P}_{\A}(K)$ are
convex, closed and have nonempty interior.
Moreover, it holds that
\be \label{Pdual=R}
\mathscr{R}_{\A}(K) = \mathscr{P}_{\A}(K)^*.
\ee
\end{prop}
\begin{proof}
Clearly, $\mathscr{R}_{\A}(K)$ and $\mathscr{P}_{\A}(K)$ are convex,
and $\mathscr{P}_{\A}(K)$ is closed.
The $K$-fullness of $\re[x]_{\A}$ implies that
there exists $p\in\mathscr{P}_{\A}(K)$ with $p>0$ on $K$.
So, $\mathscr{P}_{\A}(K)$ has nonempty interior,
since $p$ is an interior point, by Lemma~\ref{lm:int-ARz}.

We show that $\mathscr{R}_{\A}(K)$ is closed.
Let $\{y_k\}\subseteq \mathscr{R}_{\A}(K)$ be a sequence such that
$y_k \to y^* \in \re^{\A}$ as $k\to \infty$.
Note that each $\mathscr{L}_{y_k}$ is $K$-positive, i.e.,
\[
\mathscr{L}_{y_k}(p) = \langle p, y_k \rangle \geq 0 \, \quad
\forall \, p \in \mathscr{P}_{\A}(K).
\]
Letting $k\to \infty$ in the above, we get
\[
\mathscr{L}_{y^*}(p) = \langle p, y^* \rangle \geq 0 \, \quad
\forall \, p \in \mathscr{P}_{\A}(K).
\]
So, $\mathscr{L}_{y^*}$ is $K$-positive.
Since $\re[x]_{\A}$ is $K$-full, by Theorem~2.2 of \cite{FiNi2012},
we have $y^* \in \mathscr{R}_{\A}(K)$.
This implies that $\mathscr{R}_{\A}(K)$ is closed.

Next, we show that $\mathscr{R}_{\A}(K)$ has nonempty interior.
Let $z$ be the tms in the proof of Lemma~\ref{lm:int-ARz}(i).
The Riesz functional $\mathscr{L}_z$
is strictly $K$-positive. By Lemma~\ref{lm:int-ARz},
$z$ is an interior point of $\mathscr{R}_{\A}(K)$.

Last, we show that \reff{Pdual=R} is true.
Clearly, $\mathscr{R}_{\A}(K) \subseteq  \mathscr{P}_{\A}(K)^*$.
For all $y\in \mathscr{P}_{\A}(K)^*$,
$\mathscr{L}_y$ is $K$-positive. Since $\re[x]_{\A}$ is $K$-full,
by Theorem~2.2 of \cite{FiNi2012}, we have $y \in \mathscr{R}_{\A}(K)$.
So, \reff{Pdual=R} holds.
\qed
\end{proof}

\subsection{Semidefinite relaxations}
\label{sec:SDr}

First, we consider semidefinite relaxations for the cone $\mathscr{R}_{\A}(K)$.
Recall the notation $\Phi_k(g),E_k(h)$ from Subsection~\ref{sec:2.2}.
For each $k \in \N$, denote
\be \label{S:A^k(K)}
\mathscr{S}_{\A}^k(K) = \big\{ z|_{\A} :\,  z \in \Phi_k(g) \cap E_k(h) \big\}.
\ee
(If $k < \deg(\A)/2$, $\mathscr{S}_{\A}^k(K)$ is defined to be $\re^{\A}$, by default.)
Clearly, $\mathscr{R}_{\A}(K) \subseteq \mathscr{S}_{\A}^k(K)$ for all $k$.
This is because for every $y \in \mathscr{R}_{\A}(K)$,
we can always extend $y$ to a tms $z \in \mathscr{R}_{2k}(K)$ with $z|_{\A} =y$
(cf.~\cite[Prop.~3.3]{Nie-ATKMP}).
Each $\mathscr{S}_{\A}^k(K)$ is a semidefinite relaxation of $\mathscr{R}_{\A}(K)$,
since it is defined by linear matrix inequalities.
Clearly, $\mathscr{S}_{\A}^{k+1}(K) \subseteq \mathscr{S}_{\A}^k(K)$ for all $k$.
This results in the nesting containment relation:
\be \label{RAK:SDr:hier}
\mathscr{S}_{\A}^{1}(K) \supseteq  \cdots \supseteq
\mathscr{S}_{\A}^k(K)  \supseteq    \mathscr{S}_{\A}^{k+1}(K)
\supseteq \cdots  \supseteq \mathscr{R}_{\A}(K).
\ee

\begin{prop} \label{thm:cap:SkA}
Let $K \ne \emptyset$ be as in \reff{def:K}.
Suppose $I(h)+Q(g)$ is archimedean,
$\A \subseteq \N^n$ is finite and $\re[x]_{\A}$ is $K$-full.
Then, it holds that
{\small
\be \label{RAK=SDr}
\mathscr{R}_{\A}(K) = \bigcap_{k=1}^{\infty} \mathscr{S}_{\A}^k(K).
\ee
}
\end{prop}
\begin{proof}
We already know that
$\mathscr{R}_{\A}(K) \subseteq \mathscr{S}_{\A}^k(K)$ for all $k$.
So, $\mathscr{R}_{\A}(K)$ is contained in the intersection
of the right hand side of \reff{RAK=SDr}.
To prove they are indeed equal, it is enough to show that
for all $y \not \in \mathscr{R}_{\A}(K)$, we have
$y\not\in \mathscr{S}_{\A}^k(K)$ if $k$ is big enough.
Choose such an arbitrary $y$. By \reff{Pdual=R}, we know
$y \not\in \mathscr{P}_{\A}(K)^* $, and there exists
$p_1 \in \mathscr{P}_{\A}(K)$ with $\langle p_1, y \rangle < 0$.
Let $p_0 \in \re[x]_{\A}$ be such that $p_0>0$ on $K$.
Then, for $\eps >0$ small, $p_2:= p_1 + \eps p_0 >0$ on $K$ and
$\langle p_2, y \rangle < 0$. Since $I(h)+Q(g)$ is archimedean,
we have $p_2 \in Q_{k_1}(g)+I_{2k_1}(h)$ for some $k_1$, by Theorem~\ref{thm:PutThm}.
If $y \in \mathscr{S}_{\A}^{k_1}(K)$, then $y=z|_{\A}$ for some
$ z \in \Phi_{k_1}(g) \cap E_{k_1}(h)$, and we get
\[
0> \langle p_2, y \rangle = \langle p_2, z\rangle  \geq 0,
\]
a contradiction. The latter inequality is because $p_2 \in Q_{k_1}(g)+I_{2k_1}(h)$
and $\Phi_{k_1}(g) \cap E_{k_1}(h)$ is dual to $Q_{k_1}(g)+I_{2k_1}(h)$
(cf.~\reff{<p,z>=0:qmod}).
So, $y \not \in \mathscr{S}_{\A}^{k_1}(K)$,
and \reff{RAK=SDr} holds.
\qed
\end{proof}

Proposition~\ref{thm:cap:SkA} shows that
the semidefinite relaxations $\mathscr{S}_{\A}^k(K)$
can approximate $\mathscr{R}_{\A}(K)$ arbitrarily well.
Indeed, we can prove $\mathscr{S}_{\A}^k(K)$ converges to $\mathscr{R}_{\A}(K)$
if we measure their distance by normalization.
For $f \in \re[x]_{\A}$, define
\[
\mathscr{S}_{\A}^k(K,f) = \{ y \in \mathscr{S}_{\A}^k(K): \, \langle f, y \rangle = 1\},
\]
\[
\mathscr{R}_{\A}(K,f) = \{ y \in \mathscr{R}_{\A}(K): \, \langle f, y \rangle = 1\}.
\]
Define the distance
\be \label{df:dist:SDr}
dist \Big(\mathscr{S}_{\A}^k(K,f), \mathscr{R}_{\A}(K,f) \Big) =
\max_{ z \in \mathscr{S}_{\A}^k(K,f) }
\min_{ y\in \mathscr{R}_{\A}(K,f)}  \|z-y\|_2.
\ee

\begin{prop} \label{pr:dist>>0}
Let $K \ne \emptyset$ be as in \reff{def:K} and
$\A \subseteq \N^n$ be finite.
Suppose $I(h)+Q(g)$ is archimedean.
If $f \in \re[x]_{\A}$ is positive on $K$, then
\be \label{mom:dist->0}
dist \left(\mathscr{S}_{\A}^k(K,f), \mathscr{R}_{\A}(K,f) \right) \, \to \, 0
\quad \mbox{ as } \, k\to \infty.
\ee
\end{prop}
\begin{proof}
Since $f>0$ on $K$, there exists $\eps >0$ with $f-\eps >0$ on $K$.
Since $I(h)+Q(g)$ is archimedean, by Theorem~\ref{thm:PutThm},
we have $f -\eps \in Q_{N_1}(g)+I_{2N_1}(h)$ for some $N_1$.
Similarly, there exist $R>0$ and $N_2\geq N_1$
such that for all $\af \in \A$
\[
R \pm  x^{\af} \in Q_{N_2}(g) + I_{2N_2}(h).
\]
For all $y\in \mathscr{S}_{\A}^k(K,f)$ with $k \geq N_2$,
there exists $z \in \Phi_k(g) \cap E_k(h)$ such that $z|_{\A} = y$.
Since $\langle f-\eps, z \rangle \geq 0$, we get
$
\eps z_{\mathbf{0}} \leq  \langle  f, z \rangle = \langle  f, y \rangle = 1
$
and
\[
0 \leq \langle R \pm  x^{\af}, z \rangle = R z_{\mathbf{0}} \pm y_\af.
\]
(Here $\mathbf{0}$ denotes the zero vector in $\N^n$.)
This implies that $| y_{\af} | \leq R/\eps$ for all $\af \in \A$.
Hence, the sets $\mathscr{S}_{\A}^k(K,f)$, with $k \geq N_2$,
are uniformly bounded.

By \reff{RAK:SDr:hier}, we know
$dist \left(\mathscr{S}_{\A}^k(K,f), \mathscr{R}_{\A}(K,f) \right)$
is monotonically decreasing.
Suppose otherwise \reff{mom:dist->0} is not true.
Then there exists $\tau$ such that
\[
dist \left(\mathscr{S}_{\A}^k(K,f), \mathscr{R}_{\A}(K,f) \right) \geq \tau > 0
\]
for all $k$. We can select
$y^{k} \in \mathscr{S}_{\A}^{k}(K,f)$ for each $k$ such tat
\[
dist \left( y^{k}, \mathscr{R}_{\A}(K,f) \right) \geq \tau/2.
\]
The sequence $\{y^{k}\}$ is bounded, because
the sets $\mathscr{S}_{\A}^k(K,f)$ ($k \geq N_2$) are uniformly bounded,
as shown in the above. It has a convergent subsequence, say,
$y^{k_i} \to \hat{y}$ as $i\to \infty$.
Clearly, $\langle f, \hat{y} \rangle =1$ and
$dist \left( \hat{y}, \mathscr{R}_{\A}(K,f) \right) >0$.
So, $\hat{y} \not\in \mathscr{R}_{\A}(K,f)$.
This implies $\hat{y} \not\in \mathscr{R}_{\A}(K) = \mathscr{P}_{\A}(K)^*$,
by \reff{Pdual=R}. So, there exists $p_0\in \mathscr{P}_{\A}(K)$ such that
$
\langle p_0, \hat{y} \rangle < 0.
$
For a small $\eps_0>0$, we have
\[
p_1:=p_0 + \eps_0 f> 0 \quad \mbox{on} \,\, K, \qquad
\langle p_1, \hat{y} \rangle < 0.
\]
By Theorem~\ref{thm:PutThm},
we have $p_1 \in Q_{N_3}(g)+I_{2N_3}(h)$ for some $N_3$.
So, $\langle p_1, y^{k_i} \rangle \geq 0$ for all $k_i \geq N_3$.
This results in
\[
\langle p_1, \hat{y} \rangle = \lim_{ i \to \infty}
\langle p_1, y^{k_i} \rangle \geq 0,
\]
which is a contradiction. Thus, \reff{mom:dist->0} must be true.
\qed
\end{proof}

Second, we consider semidefinite relaxations for the cone
$\mathscr{R}_{\A}(K)$. Denote
\be \label{Qkrlx}
\mathscr{Q}_{\A}^k(K) = \{ p \in \re[x]_{\A} : \, p \in Q_k(g)+ E_{2k}(h) \}.
\ee
Clearly, $\mathscr{Q}_{\A}^k(K) \subseteq \mathscr{P}_{\A}(K)$ for all $k$.
Suppose $K$ is compact and $\re[x]_\A$ is $K$-full.
If $p$ is in the interior of $\mathscr{P}_{\A}(K)$,
then $p>0$ on $K$ by Lemma~\ref{lm:int-ARz},
and $p \in \mathscr{Q}_{\A}^k(K)$ for some $k$
by Theorem~\ref{thm:PutThm}, if $I(h)+Q(g)$ is archimedean.
So, we get the following proposition.

\begin{prop} \label{pr:SDr:PAK}
Let $K \ne \emptyset$ be as in \reff{def:K} and $\A \subseteq \N^n$ be finite.
Suppose $\re[x]_\A$ is $K$-full and $I(h)+Q(g)$ is archimedean. Then, we have
{\small
\be  \label{PAK:SDr}
int\left( \mathscr{P}_{\A}(K) \right) \subseteq
\bigcup_{k=1}^{\infty} \mathscr{Q}_{\A}^k(K)
\subseteq \mathscr{P}_{\A}(K).
\ee
}
\end{prop}

The second containment inequality in
\reff{PAK:SDr} generally cannot be changed to an equality.
For instance, when $K = B(0,1)$ and $\A = \N^3_6$, the Motzkin polynomial
$x_1^2x_2^2(x_1^2+x_2^2-3x_3^2)+x_3^6 \in \mathscr{P}_{\A}(K)$
but it does not belong to $\mathscr{Q}_{\A}^k(K)$ for any $k$
(cf.~\cite[Example~5.3]{Nie-jac}).

\subsection{Checking memberships}

First, checking whether $y \in \mathscr{R}_{\A}(K)$ or not
can be done by Algorithm~4.2 of \cite{Nie-ATKMP}.
It is based on solving a hierarchy of
semidefinite relaxations about moment sequences.
That algorithm has the following properties:
i) If $\re[x]_{\A}$ is $K$-full and $y \in \re^{\A}$ admits no $K$-measures,
then a semidefinite relaxation is infeasible.
This gives a certificate for the non-membership $y \not\in \mathscr{R}_{\A}(K)$.
ii) If $y$ admits a $K$-measure, then
we can asymptotically get a finitely atomic $K$-representing measure for $y$,
which certifies the membership $y \in \mathscr{R}_{\A}(K)$;
moreover, under some general conditions,
we can get such a measure within finitely many steps.

\bigskip
Second, we discuss how to check memberships in the cone $\mathscr{P}_{\A}(K)$.
Clearly, a polynomial $f \in \re[x]_{\A}$ belongs to $\mathscr{P}_{\A}(K)$
if and only if its minimum $f_{min}$ over $K$ is nonnegative.
A standard approach for computing $f_{min}$ is to
apply Lasserre's hierarchy of semidefinite relaxations ($k=1,2,\cdots$):
\be \label{las:hrky}
f_k = \max \quad \gamma \qquad s.t. \qquad f-\gamma \in I_{2k}(h) + Q_k(g).
\ee
Clearly, if $f_k \geq 0$ for some $k$, then $f \in \mathscr{P}_{\A}(K)$.
Suppose $I(h)+Q(g)$ is archimedean. For all $f \in  int(\mathscr{P}_{\A}(K))$,
we have $f_k > 0$ for some $k$, by Proposition~\ref{pr:SDr:PAK}.
For $f$ lying generically on the boundary of $\mathscr{P}_{\A}(K)$
(e.g., some standard optimality conditions hold),
we have  $f_k \geq 0$ for some $k$ (cf.~\cite{Nie-opcd}).
For the remaining non-generic cases,
it is possible that $f_k < f_{min}$ for all $k$
(cf.~\cite[Examples~5.3,\,5.6]{Nie-jac}).

Another method for computing $f_{min}$
is the Jacobian SDP relaxation in \cite{Nie-jac}.
Its basic idea is to add new polynomial equalities,
by using the Jacobian of polynomials $f, h_i,g_j$.
Suppose $\varphi := (\varphi_1,\ldots,\varphi_L)=0$ is added
(cf.~\cite[Section~2]{Nie-jac}).
Under some generic conditions on $K$ but not on $f$
(cf.~Assumption~2.2 of \cite{Nie-jac}),
$f_{min}$ equals the optimal value of
\be \label{pop:Jac-K}
\min \quad f(x) \quad
\mbox{s.t.}  \quad \varphi(x) = 0, \,  h(x) = 0, \, g(x) \geq 0.
\ee
This leads to the hierarchy of stronger semidefinite relaxations ($k=1,2,\cdots$):
\be  \label{sos:jac}
f_k^{jac} \, := \, \max \, \gamma \quad
\mbox{ s.t. } \quad f - \gamma  \in I_{2k}(h)+I_{2k}(\varphi) + Q_k(g).
\ee
An advantage of this approach is that
$\{f_k^{jac}\}$ always have finite convergence to $f_{min}$
(cf. \cite[Section~4]{Nie-jac}).
So, we can check whether $f \in \mathscr{P}_{\A}(K)$ or not
by solving finitely many semidefinite relaxations.

\section{Linear optimization problems}
\label{sec:liopt}
\setcounter{equation}{0}

Let $K$ be as in \reff{def:K} and $\A \subseteq \N^n$ be finite.
Given $a_1, \ldots, a_m,c \in \re[x]_{\A}$ and $b\in \re^m$,
we consider the linear optimization problem
\be  \label{opt:RAK}
\left\{ \baray{rl}
c^{min} := \min &  \langle c, y \rangle  \\
s.t. &  \langle a_i, y \rangle = b_i \, (i = 1,\ldots,  m),
\quad y \in \mathscr{R}_{\A}(K).
\earay \right.
\ee
The dual problem of \reff{opt:RAK} is
\be  \label{maxb:PAK}
\left\{ \baray{rl}
b^{\max}:=\max  &  b^T\lmd  \\
s.t. &  c(\lmd):= c - \sum_{i=1}^m \lmd_i a_i \in \mathscr{P}_{\A}(K).
\earay \right.
\ee
The cones $\mathscr{R}_{\A}(K)$ and $\mathscr{P}_{\A}(K)$ are hard to describe,
but they can be approximated as close as possible
by the semidefinite relaxations $\mathscr{S}_{\A}^k(K)$ in \reff{S:A^k(K)}
and $\mathscr{Q}_{\A}^k(K)$ in \reff{Qkrlx} respectively
(cf. Propositions \ref{thm:cap:SkA}, \ref{pr:SDr:PAK}).
If we relax $\mathscr{R}_{\A}(K)$ by $\mathscr{S}_{\A}^k(K)$,
then \reff{opt:RAK} is relaxed to
\be \label{mincy:k-MOM}
\left\{ \baray{rl}
c^k:= \underset{y, w}{\min} &   \langle c, y \rangle \\
s.t. &  \langle a_i, y \rangle = b_i, \, i =1,\ldots, m \\
 &   y=w|_{\A}, \,  w \in \Phi_k(g) \cap E_k(h).
\earay \right.
\ee
If $y$ is feasible in \reff{mincy:k-MOM}, then $y \in \mathscr{S}_{\A}^k(K)$.
The integer $k$ in \reff{mincy:k-MOM} is called a relaxation order.
The dual problem of \reff{mincy:k-MOM} is
\be  \label{mxblm:k-SOS}
\left\{ \baray{rl}
b^k:= \underset{\lmd =(\lmd_1, \ldots, \lmd_m) }{\max} &  b^T\lmd \\
s.t. \quad \, &  c(\lmd)  \in  Q_k(g) + I_{2k}(h).
\earay \right.
\ee
Clearly, $\lmd$ is feasible in \reff{mxblm:k-SOS}
if and only if $c(\lmd) \in \mathscr{Q}_{\A}^k(K)$.
Recall the notation $\Phi_k(g)$,
$E_k(h)$, $Q_k(g)$, $I_{2k}(h)$ from Subsection~\ref{sec:2.2}.

Clearly, we have $c^k \leq c^{min}$ and $b^k \leq b^{\max}$ for all $k$.
Let $(y^{*,k}, w^{*,k})$ be a minimizer of \reff{mincy:k-MOM},
and let $\lmd^{*,k}$ be a maximizer of \reff{mxblm:k-SOS}.
If $y^{*,k} \in \mathscr{R}_{\A}(K)$, then $c^k= c^{min}$
and $y^{*,k}$ is a minimizer of \reff{opt:RAK},
i.e., the relaxation \reff{mincy:k-MOM} is exact for solving \reff{opt:RAK}.
In such case, if $b^k = c^k$ also holds, then $b^k = b^{\max}$ and
$\lmd^{*,k}$ is a maximizer of \reff{maxb:PAK}.
If the relaxation \reff{mincy:k-MOM} is infeasible,
then \reff{opt:RAK} must also be infeasible.
Combining the above, we get the following algorithm.

\begin{alg} \label{alg:momopt}
A semidefinite algorithm for solving \reff{opt:RAK}-\reff{maxb:PAK}. \\
{\bf Input:}\, $c,a_1, \ldots, a_m \in \re[x]_{\A}$,
$b \in \re^m$ and $K$ as in \reff{def:K}. \\
\noindent
{\bf Output:} A minimizer $y^*$ of \reff{opt:RAK}
and a maximizer $\lmd^*$ of \reff{maxb:PAK},
or an answer that \reff{opt:RAK} is infeasible.

\noindent
{\bf Procedure:} \quad
\bdes

\item [Step 0] Let $k = \lceil \deg(\A)/2 \rceil$.

\item [Step 1] Solve the primal-dual pair \reff{mincy:k-MOM}-\reff{mxblm:k-SOS}.
If \reff{mincy:k-MOM} is infeasible,
stop and output that \reff{opt:RAK} is infeasible;
otherwise, compute an optimal pair $(y^{*,k},w^{*,k})$ for \reff{mincy:k-MOM}
and a maximizer $\lmd^{*,k}$ for \reff{mxblm:k-SOS}.

\item [Step 2] If $y^{*,k} \in \mathscr{R}_{\A}(K)$,
then $y^{*,k}$ is a minimizer of \reff{opt:RAK};
if in addition $b^k = c^k$,
then $\lmd^{*,k}$ is a maximizer of \reff{maxb:PAK};
stop and output $y^*=y^{*,k}$, $\lmd^*=\lmd^{*,k}$.
Otherwise, let $k:=k+1$ and go to Step~1.

\edes

\end{alg}

\begin{remark}
Checking if $y^{*,k} \in \mathscr{R}_{\A}(K)$ or not
is a stopping criterion for Algorithm~\ref{alg:momopt}.
If there exists $t \geq \deg(\A)/2$ such that
$w^{*,k}|_{2t}$ is flat, then $y^{*,k} \in \mathscr{R}_{\A}(K)$.
This gives a convenient way to terminate the algorithm.
It is possible that $y^{*,k}$ belongs to $\mathscr{R}_{\A}(K)$
while $w^{*,k}|_{2t}$ is not flat for all $t$
(cf.~Example~\ref{exm:4.7}).
In such case, we can apply Algorithm~4.2 of \cite{Nie-ATKMP}
to check if $y^{*,k} \in \mathscr{R}_{\A}(K)$ or not.
\end{remark}

Feasibility and infeasibility issues of
\reff{opt:RAK}-\reff{maxb:PAK} are more delicate.
They will be studied separately in Section~\ref{sec:feas}.
In Section~4.1, we prove the asymptotic and finite convergence
of Algorithm~\ref{alg:momopt}.
In Section~4.2, we present some examples.

\subsection{Convergence analysis}

First, we prove the asymptotic convergence of Algorithm~\ref{alg:momopt}.
%
%

\begin{theorem}  \label{thm:alg:asymp}
Let $K$ be as in \reff{def:K} and $\A \subseteq \N^n$ be finite.
Suppose \reff{opt:RAK} is feasible, \reff{maxb:PAK} has an interior point,
$\re[x]_{\A}$ is $K$-full, and $Q(g)+I(h)$ is archimedean. Then, we have:
\bit

\item [(i)] For all $k$ sufficiently large,
\reff{mxblm:k-SOS} has an interior point and
\reff{mincy:k-MOM} has a minimizing pair $(y^{*,k},w^{*,k})$.

\item [(ii)] The sequence $\{y^{*,k}\}$
is bounded, and each of its accumulation points is
a minimizer of \reff{opt:RAK}.

\item [(iii)] The sequence $\{b^k\}$ converges to
the maximum $b^{\max}$ of \reff{maxb:PAK}.

\eit
\end{theorem}

\begin{remark}
For the classical case $\A = \N_d^n$,
if one of $c,a_1,\ldots, a_m$ is positive on $K$,
Lasserre \cite[Theorem~1]{Las08} proved
$c^k \to c^{min}$ as $k \to \infty$.
This can also be implied by the item (ii) of Theorem~\ref{thm:alg:asymp},
because $c^k = \langle c, y^{*,k} \rangle$.
The conclusion that each accumulation point of $\{y^{*,k}\}$
is a minimizer of \reff{opt:RAK} is a stronger property.
Moreover, the assumption that \reff{maxb:PAK} has an interior point
is weaker than that one of $c,a_1,\ldots, a_m$ is positive on $K$.
Theorem~\ref{thm:alg:asymp} discusses
more general cases of $\A$.
\end{remark}

\begin{proof}[of Theorem~\ref{thm:alg:asymp}]\,
(i) Let $\lmd^0$ be an interior point of \reff{maxb:PAK}.
Then $c(\lmd^0) = c - \sum_{i=1}^m \lmd_i^0 a_i >0$ on $K$, by Lemma~\ref{lm:int-ARz}.
The archimedeanness of $I(h)+Q(g)$ implies that $K$ is compact.
So, there exist $\eps_0>0$ and $\theta > 0$ such that
\[
c(\lmd) - \eps_0 > \eps_0 \quad \forall \,  \lmd \in B(\lmd^0,\theta).
\]
By Theorem~6 of \cite{NS07}, there exists $N_0>0$ such that
\[
c(\lmd) - \eps_0 \in I_{2N_0}(h) + Q_{N_0}(g)
\quad \forall \, \lmd \in B(\lmd^0,\theta).
\]
So, \reff{mxblm:k-SOS} has an interior point
for all $k\geq N_0$, and the strong duality holds between
\reff{mincy:k-MOM} and \reff{mxblm:k-SOS}.
Since \reff{opt:RAK} is feasible, the relaxation \reff{mincy:k-MOM}
is also feasible and has a minimizing pair $(y^{*,k}, w^{*,k})$
(cf. \cite[Theorem~2.4.I]{BTN}).

(ii) First, we show that $\{y^{*,k}\}$ is a bounded sequence.
Let $c(\lmd^0)$ and $\eps_0$ be as in the proof of (i).
The set $I_{2N_0}(h) + Q_{N_0}(g)$ is dual to
$E_{2N_0}(h)\cap \Phi_{N_0}(g)$.
For all $k\geq N_0$, we have
$w^{*,k} \in E_{N_0}(h) \cap \Phi_{N_0}(g)$ and
\[
0 \leq \langle c(\lmd^0) - \eps_0, w^{*,k} \rangle =
 \langle c(\lmd^0) , w^{*,k} \rangle - \eps_0 \langle 1, w^{*,k} \rangle,
\]
\[
 \langle c(\lmd^0) , w^{*,k} \rangle  =  \langle c , w^{*,k} \rangle
 - \sum_{i=1}^m \lmd_i^0 \langle a_i , y^{*,k} \rangle
 = \langle c , w^{*,k} \rangle - b^T \lmd^0.
\]
Since $\langle c , w^{*,k} \rangle \leq c^{min}$, it holds that
\[
 \langle c(\lmd^0) , w^{*,k} \rangle  \leq T_0:= c^{min} - b^T \lmd^0.
\]
Combining the above, we get (denote by $\mathbf{0}$ the zero vector in $\N^n$)
\[
0 \leq \langle c(\lmd^0) - \eps_0, w^{*,k} \rangle
\leq  T_0 - \eps_0 (w^{*,k})_{\mathbf{0}},
\]
\[
(w^{*,k})_{\mathbf{0}} \leq T_1:=T_0/\eps_0.
\]
Since $I(h)+Q(g)$ is archimedean,
there exist $\rho>0$ and $k_1 \in \N$ such that
\[
\rho - \|x\|_2^2  \in I_{2k_1}(h) + Q_{k_1}(g).
\]
So, for all $k\geq k_1$, we get
{\small
\[
0 \leq \langle \rho - \|x\|_2^2, w^{*,k} \rangle
%
%
= \rho (w^{*,k})_{\mathbf{0}} - \sum_{|\af|=1} (w^{*,k})_{2\af}, \quad
\sum_{|\af|=1} (w^{*,k})_{2\af} \leq  \rho T_1.
\]
}
For each $t = 1, \ldots, k-k_1$, we have
\[
\|x\|_2^{2t-2}(\rho - \|x\|_2^2)  \in I_{2k}(h) + Q_{k}(g).
\]
The membership $w^{*,k} \in \Phi_k(g) \cap I_k(h)$ implies that
\[
\rho \langle \|x\|_2^{2t-2}, w^{*,k} \rangle - \langle \|x\|_2^{2t}, w^{*,k} \rangle
\geq 0, \quad t = 1, \ldots, k-k_1.
\]
The above then implies that
\[
\langle \|x\|_2^{2t}, w^{*,k} \rangle  \leq \rho^{t} T_1,
\quad t = 1, \ldots, k-k_1.
\]
Let $z^k := w^{*,k}|_{2k-2k_1}$, then
the moment matrix $M_{k-k_1}(z^k) \succeq 0$ and
{\small
\[
\|z^k\|_2 \leq \|M_{k-k_1}(z^k)\|_F \leq Trace(M_{k-k_1}(z^k)) =
\sum_{i=0}^{k-k_1} \sum_{|\af| = i}\, (w^{*,k})_{2\af},
\]
\[
\sum_{|\af| = i}\, (w^{*,k})_{2\af} = \langle \sum_{|\af| = i} \,x^{2\af}, z^k \rangle
 \leq \langle \|x\|_2^{2i}, z^{k} \rangle \leq  \rho^i T_1.
\]
}
\noindent
The above then implies that
\[
\|z^k\|_2 \leq  (1 + \rho + \cdots + \rho^{k-k_1}) T_1.
\]
Fix $k_2>k_1$ such that $y^{*,k}$ is a subvector of
$z^k|_{k_2-k_1}$. From $y^{*,k} = z^k|_{\A}$, we get
\[
\| y^{*,k}\|_2 \leq  \|z^k\|_2  \leq  (1 + \rho + \cdots + \rho^{k_2-k_1}) T_1.
\]
This shows that the sequence $\{y^{*,k}\}$ is bounded.

Second, we show that every accumulation point of $\{y^{*,k}\}$
is a minimizer of \reff{opt:RAK}.
Let $y^*$ be such an arbitrary one.
We can generally further assume $y^{*,k} \to y^*$ as $k\to\infty$.
We need to show that $y^*$ is a minimizer of \reff{opt:RAK}.
Since $K$ is compact, by the archimedeanness of $I(h)+Q(g)$,
we can generally assume $K \subseteq B(0,\rho)$ with $\rho<1$,
up to a scaling. In the above, we have shown that
\[
\|z^k\|_2 \leq   T_1/(1-\rho).
\]
This implies that the sequence $\{z^k\}$ is bounded.
Each tms $z^k$ can be extended to a vector in
$\re^{\N_{\infty}^n}$ by adding zero entries to the tailing.
The set $\re^{\N_{\infty}^n}$ is a Hilbert space, equipped with the inner product
{\small
\[
\langle u, v \rangle := \sum_{ \af \in \N^n } u_\af v_\af,
\quad \forall \, u, v \in \re^{\N_{\infty}^n}.
\]
}\noindent
So, the sequence $\{ z^k \}$ is also bounded in $\re^{\N_{\infty}^n}$.
By Alaoglu's Theorem (cf.~\cite[Theorem~V.3.1]{Conway} or \cite[Theorem~C.18]{LasBok}),
it has a subsequence $\{z^{k_j} \}$ that is convergent in the weak-$\ast$ topology.
That is, there exists $z^* \in \re^{\N_{\infty}^n}$ such that
\[
\langle f, z^{k_j} \rangle \, \to \,  \langle f,  z^* \rangle
 \quad \mbox{ as } j \to \infty
\]
for all $f \in \re^{\N_{\infty}^n}$. Clearly, this implies that
for each $\af \in \N^n$
\be \label{2k2t->w*2t}
(z^{k_j})_{\af} \to  (z^*)_{\af}.
\ee
Since $z^k|_{\A} = y^{*,k} \to y^*$, we get $z^*|_{\A} = y^*$.
Note that
$
z^{k_j} \in \Phi_{k_j}(g) \cap E_{k_j}(h)
$
for all $j$. For each $r=1,2,\ldots$,
if $k_j \geq 2r$, then (cf.~Section~2.1)
\[
L_{h_i}^{(r)}(z^{(k_j)}) = 0\,(1\leq i \leq m_1), \quad
L_{g_i}^{(r)}(z^{(k_j)})\succeq 0 \,(0\leq i \leq m_2).
\]
Hence, \reff{2k2t->w*2t} implies that for all $r=1,2,\ldots$
\[
L_{h_i}^{(r)}(z^*) = 0\,(1\leq i \leq m_1), \quad
L_{g_i}^{(r)}(z^*)\succeq 0 \,(0\leq i \leq m_2).
\]
This means that $z^* \in \re^{\N_{\infty}^n}$ is a full moment sequence
whose localizing matrices of all orders are positive semidefinite.
By Lemma~3.2 of Putinar \cite{Put}, $z^*$ admits a $K$-measure.
Clearly, $\langle a_i, y^*\rangle = b_i$ for all $i$.
So, $y^* = z^*|_{\A}$ is feasible for \reff{opt:RAK}
and $c^{min} \leq  \langle c, y^*\rangle$.
Because \reff{mincy:k-MOM} is a relaxation of \reff{opt:RAK}
and $w^{*,k}$ is a minimizer of \reff{mincy:k-MOM}, it holds that
\[
c^{min} \geq \langle c, y^{*,k}\rangle, \quad k=1,2,\ldots
\]
Hence, we get
\[
c^{min} \geq \lim_{k\to \infty} \langle c, y^{*,k}\rangle = \langle c, y^*\rangle.
\]
Therefore, $c^{min} = \langle c, y^*\rangle$ and
$y^*$ is a minimizer of \reff{opt:RAK}.

\medskip
(iii) For each $\eps>0$, there exists $\lmd^\eps$ such that
$c(\lmd^\eps) \in \mathscr{P}_{\A}(K)$ and
\[ b^{max} - \eps <  b^T \lmd^\eps \leq b^{max} . \]
Let $\lmd^0$ be as in the proof of item (i),
and let $\lmd(\eps) = (1-\eps) \lmd^\eps + \eps \lmd^0$.
Then $c(\lmd(\eps))>0$ on $K$ and
\[
b^T\lmd(\eps) = (1-\eps) b^T\lmd^\eps + \eps b^T\lmd^0
> (1-\eps) (b^{max} - \eps) + \eps b^T\lmd^0.
\]
By Theorem~\ref{thm:PutThm}, if $k$ is big enough,
then $c(\lmd(\eps)) \in I_{2k}(h)+Q_k(g)$ and
\[
b^k > (1-\eps) (b^{max} - \eps) + \eps b^T\lmd^0.
\]
Since $b^k \leq b^{max}$ for all $k$,
we get $b^k \to b^{max}$ as $k \to \infty$.
\qed
\end{proof}

Second, we prove the finite convergence of Algorithm~\ref{alg:momopt}
under a general assumption.

\begin{ass}  \label{as:c*inQ+KT}
Suppose $\lmd^*$ is a maximizer of \reff{maxb:PAK}
and $c^*:=c(\lmd^*)$ satisfies:
\bit
\item [(i)] There exists $k_1 \in \N$ such that
$c^* \in I_{2k_1}(h)+Q_{k_1}(g)$;

\item [(ii)] The optimization problem
\be \label{min-c*-K}
\min \quad c^*(x) \quad s.t. \quad h(x) = 0, \, g(x) \geq 0
\ee
has finitely many KKT points $u$ with $c^*(u)=0$.

\eit
\end{ass}

\begin{theorem} \label{thm:mom:ficvg}
Let $K$ be as in \reff{def:K}.
Suppose \reff{opt:RAK} is feasible, \reff{maxb:PAK} has an interior point,
$\re[x]_{\A}$ is $K$-full and Assumption~\ref{as:c*inQ+KT} holds.
If $w^{*,k}$ is optimal for \reff{mincy:k-MOM},
then $w^{*,k}|_{2t}$ is flat for all $k > t$ big enough.
\end{theorem}

\begin{remark}
If $c^*$ is generic on the boundary of the cone $\mathscr{P}_{\A}(K)$, then
Assumption~\ref{as:c*inQ+KT} holds (cf.~\cite{NR09,Nie-opcd}).
For instance, if some standard
optimality conditions are satisfied for
the optimization problem \reff{min-c*-K},
then Assumption~\ref{as:c*inQ+KT} is satisfied.
Theorem~\ref{thm:mom:ficvg} implies that
Algorithm~\ref{alg:momopt} generally converges in finitely many steps.
This fact has been observed in numerical experiments.
\end{remark}

\begin{proof}[of Theorem~\ref{thm:mom:ficvg}]\,
The existence of a minimizer $(y^{*,k},w^{*,k})$
is shown in Theorem~\ref{thm:alg:asymp}.
Because \reff{opt:RAK} is feasible and \reff{maxb:PAK} has an interior point,
\reff{opt:RAK} has a minimizer $y^*$
and there is no duality gap between \reff{opt:RAK} and \reff{maxb:PAK}, i.e.,
\[
0 = \langle c, y^* \rangle - b^T\lmd^* =
\langle c^*, y^* \rangle.
\]
Clearly, $c^* \geq 0$ on $K$.
Let $\mu^*$ be a $K$-representing measure for $y^*$. Then,
every point in $\supp{\mu^*}$ is a minimizer of \reff{min-c*-K},
and the minimum value is $0$.
The $k$-th Lasserre's relaxation for \reff{min-c*-K} is
(cf. \cite{Las01,Nie-FT})
\be \label{sos:minc*}
\gamma_k := \max  \quad \gamma \quad s.t. \quad
c^*-\gamma \in  I_{2k}(h) + Q_k(g).
\ee
Then, $\gamma_k=0$ for all $k\geq k_1$.
The sequence $\{\gamma_k\}$ has finite convergence.
The relaxation \reff{sos:minc*} achieves its optimal value for all $k \geq k_1$,
by Assumption~\ref{as:c*inQ+KT} (i).
The dual problem of \reff{sos:minc*} is
\be \label{mom:minc*}
\min_{ w }  \quad \langle c^*, w \rangle \quad s.t. \quad
w \in \Phi_k(g) \cap E_k(h), w_{\mathbf{0}} = 1.
\ee
By Assumption~\ref{as:c*inQ+KT}(ii),
\reff{min-c*-K} has only finitely many critical points on which $c^*=0$.
So, Assumption~2.1 in \cite{Nie-FT} for the problem \reff{min-c*-K} is
satisfied\footnote{In \cite{Nie-FT}, optimization problems with only
inequalities were discussed. If there are equality constraints,
Assumption~2.1 in \cite{Nie-FT} can be naturally modified to include all
equalities, and the conclusion of Theorem~2.2 of \cite{Nie-FT} is still true,
with the same proof.}.
Suppose $w^{*,k}$ is optimal for \reff{mincy:k-MOM}.

If $(w^{*,k})_{\mathbf{0}}=0$, then $vec(1)^T M_k(w^{*,k}) vec(1) =0$
and $M_k(w^{*,k}) vec(1) =0$, because $M_k(w^{*,k})\succeq 0$.
(Here $vec(p)$ denotes the coefficient vector of a polynomial $p$.)
This implies that $M_k(w^{*,k}) vec(x^\af)=0$ for all $|\af| \leq k-1$
(cf.~\cite[Lemma~5.7]{Lau}). For all $|\af|\leq 2k-2$,
we can write $\af = \bt + \eta$ with $|\bt|,|\eta| \leq k-1$,
and get
\[
(w^{*,k})_\af= vec(x^\bt)^T M_k(w^{*,k}) vec(x^\eta) =0.
\]
So, the truncation $w^{*,k}|_{2k-2}$ is flat.

If $(w^{*,k})_{\mathbf{0}}>0$, we can scale $w^{*,k}$
such that $(w^{*,k})_{\mathbf{0}}=1$.
Then $w^{*,k}$ is a minimizer of \reff{mom:minc*}
because $\langle c^*, w^{*,k} \rangle = 0$ for all $k\geq k_1$.
By Theorem~2.2 of \cite{Nie-FT},
$w^{*,k}$ has a flat truncation $w^{*,k}|_{2t}$ for some $t \geq \deg(\A)/2$,
for all $k$ sufficiently large.
\qed
\end{proof}

\subsection{Some examples}

Semidefinite relaxations \reff{mincy:k-MOM} and \reff{mxblm:k-SOS}
can be solved by {\tt GloptiPoly~3} \cite{GloPol3}.

\begin{exm}
Let $K$ be the simplex $\Dt_n = \{x\in \re_+^n: x_1+\cdots+x_n=1\}$
and $\A = \{\af \in \N^n: |\af| = 2\}$.
Then $\mathscr{P}_{\A}(\Dt_n)$ is the cone
of $n\times n$ copositive matrices (denoted as $\mathtt{Co}(n)$),
and $\mathscr{R}_{\A}(\Dt_n)$ is the cone
of $n\times n$ completely positive matrices (denoted as $\mathtt{Cp}(n)$).
The simplex $\Dt_n$ is defined by the tuples
$h=(x_1+\cdots+x_n - 1)$ and $g=(x_1,\ldots,x_n)$, as in \reff{def:K}.\\
(i) Let $c = (x_1+\cdots+x_6)^2$
and $a_1 = x_1x_2-x_2x_3+x_3x_4-x_5x_6+x_6x_1$.
We want to know the maximum $\lmd$ such that
$c  - \lmd a_1 \in \mathtt{Co}(6)$.
We formulate this problem in the form \reff{maxb:PAK}
and then solve it by Algorithm~\ref{alg:momopt}.
For $k=2$, $y^{*,2} \in \mathtt{Cp}(6)$ (because it admits the measure
$4 \dt_{(1/2,1/4,0,0,0, 1/4)}$) and $\lmd^{*,2} = 4$.
Since $c^k=b^k$ for $k=2$,
we know the maximum $\lmd$ in the above is $4$. \\
(ii) Consider the matrix
\[
C = \bbm
 * &  1  &  2  & 3  & 4  \\
 1 &  *  &  1  & 2  & 3  \\
 2 &  1  &  *  & 1  & 2  \\
 3 &  2  &  1  & *  & 1  \\
 4 &  3  &  2  & 1  & *  \\
\ebm
\]
where each $*$ means the entry is not given.
We want to know the smallest trace of $C$
for $C \in \mathtt{Cp}(5)$.
We formulate this problem in the form \reff{opt:RAK}
%
%
and then solve it by Algorithm~\ref{alg:momopt}.
For $k=2$, $y^{*,2} \in \mathscr{R}_{\A}(\Dt_5)$
(verified by Algorithm~4.2 of \cite{Nie-ATKMP}).
So, the minimum trace of $C \in \mathtt{Cp}(5)$
is $20.817217,$\footnote{Throughout the paper,
six decimal digits are shown for numerical results.}
while the diagonal entries $C_{11},\ldots,C_{55}$ are
$6.031873$, $3.968627$, $0.816217$, $3.968627$, $6.031873$ respectively.
\qed
\end{exm}

\begin{exm} \label{exm:4.7}
Let $K= B(0,1)$ be the unit ball in $\re^2$ and $\A = \N_6^4$.
We want to know the maximum $\lmd_1 + \lmd_2$ such that
{\small
\[
x_1^4x_2^2+6x_1^2x_2^2+4x_1x_2^4+x_2^6+x_2^2 -
\lmd_1 ( x_1^3x_2^2+x_1x_2^2 ) -
\lmd_2 ( x_1^2x_2^4+x_2^4)  \in \mathscr{P}_{\A}(K).
\]
\noindent}We formulate
this problem in the form \reff{maxb:PAK}
and then solve it by Algorithm~\ref{alg:momopt}.
When $k=3$, $y^{*,3} \in \mathscr{R}_{\A}(K)$
(verified by Algorithm~4.2 of \cite{Nie-ATKMP}),
and $\lmd^{*,3}=(4,2)$. Since $c^k=b^k$ for $k=2$,
the optimal $(\lmd_1,\lmd_2)$ in the above is $(4,2)$.
\qed
\end{exm}

\begin{exm}
Let $K= \mathbb{S}^2$ and $\A = \{\af \in \N^3: |\af| = 6\}$.
Then, $\mathscr{P}_{\A}(K)=P_{3,6}$
and $\mathscr{R}_{\A}(K)$ is the cone of sextic tms'
admitting measures supported in $\mathbb{S}^2$. \\
%
%
(i) The form
$c = x_1^6+x_2^6+x_3^6$ lies in the interior of $P_{3,6}$. Let
\[
a_1 = x_1^2x_2^4+ x_2^2x_3^4 + x_3^2 x_1^4, \,
a_2 = x_1^3x_2^3+ x_2^3x_3^3 + x_3^3 x_1^3, \,
a_3 = x_1^5x_2+ x_2^5x_3 + x_3^5 x_1.
\]
We want to know the maximum $\lmd_1+\lmd_2+\lmd_3$ such that
\[
c  - \lmd_1 a_1 - \lmd_2 a_2 - \lmd_3 a_3  \in P_{3,6}.
\]
We formulate this problem in the form \reff{maxb:PAK}
and then solve it by Algorithm~\ref{alg:momopt}.
When $k=3$, $y^{*,3} \in \mathscr{R}_{\A}(K)$
(it admits the measure $9 \dt_{(1,1,1)/\sqrt{3}}$), and
\[
%
%
\lmd^{*,3}=( -1.440395,   2.218992,  0.221403).
\]
Since $c^k=b^k$ for $k=2$, we know $\lmd^{*,3}$ is also optimal for the above.
\\
(ii) We want to know the minimum value of
$\int (x_1^6+x_2^6+x_3^6) \mathtt{d} \mu$ for all
measures $\mu$ supported in $\mathbb{S}^2$ such that
{\small
\[
\int x_1^3x_2^3 \mathtt{d} \mu =
\int x_2^3x_3^3 \mathtt{d} \mu =
\int x_3^3x_1^3 \mathtt{d} \mu,
\int x_1^2x_2^2x_3^2 \mathtt{d} \mu = 1,
\int (x_1^4x_2^2 + x_2^4x_3^2 + x_3^4x_1^2) \mathtt{d} \mu = 3.
\]
}\noindent
We formulate the problem in the form \reff{opt:RAK}
%
%
and then solve it by Algorithm~\ref{alg:momopt}.
When $k=3$, $y^{*,3} \in \mathscr{R}_{\A}(K)$
because it admits the measure
\[
\frac{27}{4} \left( \dt_{(1,1,1)/\sqrt{3}} +  \dt_{(-1,1,1)/\sqrt{3}} +
\dt_{(1,-1,1)/\sqrt{3}} + \dt_{(1,1,-1)/\sqrt{3}} \right).
\]
So, the minimum of
$\int (x_1^6+x_2^6+x_3^6) \mathtt{d} \mu$
for $\mu$ satisfying the above is $3$.
\qed
\end{exm}

If a linear optimization problem
with cone $\mathscr{R}_{\A}(K)$ is given in the form \reff{maxb:PAK},
it can also be equivalently formulated in the form \reff{opt:RAK}.
For instance, given $z_0, \ldots, z_m \in \re^{\A}$
and $\ell = (\ell_1,\ldots, \ell_m) \in \re^m$, consider the problem
\be  \label{mxby:RAK}
\left\{ \baray{rl}
\max &   \ell_1 \lmd_1 + \cdots + \ell_m \lmd_m  \\
s.t. &  z_0 - \lmd_1 z_1 - \cdots - \lmd_m z_m \in \mathscr{R}_{\A}(K).
\earay \right.
\ee
Let $\{p_1, \ldots, p_r\}$ be a basis of
the orthogonal complement of $\mbox{span}\{z_1,\ldots,z_m\}$.
Then, $y \in z_0 + \mbox{span}\{z_1,\ldots,z_m\}$ if and only if
\[
p_1^T y = p_1^T z_0, \ldots, p_m^T y = p_m^T z_0.
\]
We can consider each $p_i$ as a polynomial in $\re[x]_{\A}$.
Let $Z = \bbm z_1  & \cdots & z_m \ebm$. Assume $\rank(Z) = m$.
If $y = z_0 -Z \lmd$, then
\[
\lmd = (Z^TZ)^{-1} Z^T(z_0-y).
\]
Let $p_0$ be a polynomial in $\re[x]_{\A}$ such that
\[
\langle p_0, y \rangle =  \ell^T (Z^TZ)^{-1} Z^T y = \ell^T \lmd.
\]
Then \reff{mxby:RAK} is equivalent to
\be  \label{min:p0y:RAK}
\left\{ \baray{rl}
\min &   \langle p_0, y \rangle \\
s.t. &  \langle p_i, y \rangle = p_i^T z_0\,(i=1,\ldots,m),
\quad y \in  \mathscr{R}_{\A}(K).
\earay \right.
\ee
If $y^*$ is a minimizer of \reff{min:p0y:RAK}, then
\[
\lmd^* = (Z^TZ)^{-1} Z^T(z_0-y^*)
\]
is a maximizer of \reff{mxby:RAK}.
Similarly, every linear optimization problem with cone $\mathscr{P}_{\A}(K)$,
which is given in the form \reff{opt:RAK},
can also be formulated like \reff{maxb:PAK}.

\begin{exm}
Let $K = \mathbb{S}^{n-1}$ and $\A = \{\af \in \N^n: |\af|=d\}$
($d$ is even). Then $\mathscr{R}_{\A}(K)$ is equivalent to $Q_{n,d}$,
the cone of sums of $d$-th power of real linear forms in $n$ variables
(cf. \cite[Sec.~6.2]{Nie-ATKMP}). \\
(i) The sextic form $(x_1^2+x_2^2+x_3^2)^3$ belongs to $Q_{3,4}$ (cf. \cite{Rez92}).
We want to know the maximum $\lmd$ such that
\[
(x_1^2+x_2^2+x_3^2)^3 - \lmd (x_1^6+x_2^6+x_3^6)  \in Q_{3,6}.
\]
The problem is equivalent to finding the biggest $\lmd$ such that
$
z_0 - \lmd z_1 \in \mathscr{R}_{\A}(K),
$
where $z_0,z_1$ are tms' whose entries are zeros except
{\small
\[
(z_0)_{(6,0,0)} = (z_0)_{(0,6,0)} = (z_0)_{(0,0,6)} = 1, \quad
(z_0)_{(2,2,2)} = 1/15,
\]
\[
(z_0)_{(4,2,0)} = (z_0)_{(2,4,0)} = (z_0)_{(0,4,2)} = (z_0)_{(0,2,4)} =
(z_0)_{(4,0,2)} = (z_0)_{(2,0,4)} = 1/5,
\]
\[
(z_1)_{(6,0,0)} = (z_1)_{(0,6,0)} = (z_1)_{(0,0,6)} = 1.
\]
\noindent}We formulate
this problem in the form \reff{min:p0y:RAK}
and then solve it by Algorithm~\ref{alg:momopt}.
For $k=4$, $y^{*,4} \in \mathscr{R}_{\A}(K)$
(verified by Algorithm~4.2 of \cite{Nie-ATKMP}), and $\lmd^{*,4}=2/3$.
Since $c^k=b^k$ for $k=4$,
the maximum $\lmd$ in the above is $2/3$,
which confirms the result of Reznick \cite[p.~146]{Rez92}.

\noindent
(ii) We want to know the maximum
$\lmd_1+\lmd_2$ such that
\[
(x_1^2+x_2^2+x_3^2)^3 - \lmd_1 (x_1^3x_2^3+x_2^3x_3^3+x_3^3x_1^3)
- \lmd_2 x_1^2x_2^2x_3^2  \in Q_{3,6}.
\]
The problem is equivalent to
\[
\max \quad \lmd_1+\lmd_2 \quad s.t. \quad
z_0 - \lmd_1 z_1 - \lmd_2 z_2 \in \mathscr{R}_{\A}(K)
\]
where $z_0$ is same as in (i) and $z_1,z_2$ are tms'
whose entries are zeros except
\[
(z_1)_{(3,3,0)} = (z_1)_{(3,0,3)} = (z_1)_{(0,3,3)} = 1/20, \quad
(z_2)_{(2,2,2)} = 1/90.
\]
We formulate this problem in the form \reff{min:p0y:RAK}
and solve it by Algorithm~\ref{alg:momopt}.
For $k=3$, $y^{*,3} \in \mathscr{R}_{\A}(K)$
(verified by Algorithm~4.2 of \cite{Nie-ATKMP}), and $\lmd^{*,3}=(2,6)$.
Since $c^k=b^k$ for $k=3$, we know
the optimal $\lmd$ in the above is $(2,6)$.
The SOEP decomposition of the polynomial
$
(x_1^2+x_2^2+x_3^2)^3 - 2(x_1^3x_2^3+x_2^3x_3^3+x_3^3x_1^3)
- 6 x_1^2x_2^2x_3^2
$
is
{\small \small
\[
\frac{7}{50} \sum_{ 1\leq i < j\leq 3 } (x_i - x_j )^6 +
\sum_{ 1\leq i \ne j\leq 3 }
\left(
\left( \frac{1}{\sqrt{10}}-\frac{\sqrt{2}}{5} \right)^{1/3} x_i +
\left( \frac{1}{\sqrt{10}}+\frac{\sqrt{2}}{5} \right)^{1/3} x_j
\right)^6.
\]
}
\qed
\end{exm}

\section{Feasibility and infeasibility}
\label{sec:feas}
\setcounter{equation}{0}

A basic question in linear optimization is to check
whether a cone intersects an affine subspace or not.
For the cones $\mathscr{R}_{\A}(K)$ and $\mathscr{P}_{\A}(K)$,
this question is about checking whether the optimization problems
\reff{opt:RAK} and \reff{maxb:PAK} are feasible or not.
If they are feasible, we want to get a feasible point;
if they are not, we want a certificate for the infeasibility.

\subsection{Finding feasible points}

First, we discuss how to check whether \reff{opt:RAK} is feasible or not.
Suppose $a_1,\ldots,a_m \in \re[x]_{\A}$ and $b\in\re^m$ are given
as in \reff{opt:RAK}, while the objective $c$ is not necessarily given.
Generally, we can assume $\re[x]_{\A}$ is $K$-full.
$\big($Otherwise, if $\re[x]_{\A}$ is not $K$-full,
let $\A^\prm := \A \cup \{0\}$, then $\re[x]_{\A^\prm}$
is always $K$-full because $1 \in \re[x]_{\A^\prm}$.
Since $a_1,\ldots,a_m \in \re[x]_{\A}$,
\reff{opt:RAK} is feasible if and only if
there exists $w \in \re^{\A^\prm}$ satisfying
\be \label{aw=b:extn}
\langle a_i, w \rangle = b_i \,( 1 \leq i \leq m),
\quad w \in \mathscr{R}_{\A^\prm}(K).
\ee
This is because: (i) if $y$ is feasible for \reff{opt:RAK},
then $y$ can be extended to a tms $w \in \mathscr{R}_{\A^\prm}(K)$,
i.e., $y = w|_{\A}$ (cf.~\cite[Prop.~3.3]{Nie-ATKMP}),
and such $w$ satisfies \reff{aw=b:extn};
(ii) if $w$ satisfies \reff{aw=b:extn}, then
the truncation $y = w|_{\A}$ satisfies \reff{opt:RAK}.$\big)$

Choose $c \in \re[x]_{\A}$ such that $c>0$ on $K$.
Consider the resulting optimization problems \reff{opt:RAK} and \reff{maxb:PAK}.
For such $c$, the dual problem \reff{maxb:PAK} has
an interior point. We can apply Algorithm~\ref{alg:momopt}
to solve \reff{opt:RAK}-\reff{maxb:PAK}.
If \reff{opt:RAK} is feasible, we can get a feasible point of \reff{opt:RAK}.

\begin{exm}
Let $\A =\N_6^3$ and $K=[-1, 1]^3$ be the unit cube,
which is defined by $h=(0)$ and $g=(1-x_1^2,1-x_2^2,1-x_3^2)$.
We want to know whether there exists a measure $\mu$
supported in $[-1,1]^3$ such that
{\small
\[
\int (x_1 x_2 + x_2x_3 + x_3x_1)\mathtt{d} \mu = 0, \,
\int (x_1^2x_2^2+x_2^2x_3^2+x_3^2x_1^2) \mathtt{d} \mu = 1, \,
\]
\[
\int (x_1^3x_2^2 + x_2^3x_3^2 + x_3^3x_1^2) \mathtt{d} \mu = 1.
\]
}\noindent
Let $a_1,a_2,a_3$ be the polynomials inside the above integrals respectively.
This problem is equivalent to whether there exists
$y \in \mathscr{R}_{\A}([-1,1]^3)$ satisfying
\[
\langle a_1, y \rangle = 0, \quad
\langle a_2, y \rangle = 1, \quad
\langle a_3, y \rangle = 1.
\]
Choose $c = \sum_{ 0 \leq |\af| \leq 3} x^{2\af}$.
For $k=3$, $y^{*,3}$ admits the measure
$
\frac{1}{2} \dt_{(0, 1, -1)} +
\frac{1}{6} \dt_{(1,1,1)},
$
which satisfies the above.
\qed
\end{exm}

Second, we discuss how to check whether \reff{maxb:PAK} is feasible.
Suppose $c,a_1,\ldots,a_m \in \re[x]_{\A}$ are given,
while $b$ is not necessarily. Let $k=\lceil \deg(\A)/2 \rceil$.
Solve the semidefinite feasibility problem
\be  \label{ck:f-in-QIk}
c -\lmd_1 a_1  - \cdots - \lmd_m a_m \in Q_k(g) + I_{2k}(h).
\ee
If \reff{ck:f-in-QIk} is feasible,
we can get a feasible point of \reff{maxb:PAK};
if not, let $k:=k+1$ and solve \reff{ck:f-in-QIk} again.
Repeat this process.
If the affine subspace $c + \mbox{span}\{a_1,\ldots,a_m\}$ intersects
the interior of $\mathscr{P}_{\A}(K)$,
we can always find a feasible point of \reff{maxb:PAK}
by solving \reff{ck:f-in-QIk}.
This can be implied by Proposition~\ref{pr:SDr:PAK}, under the archimedeanness.
If $c + \mbox{span}\{a_1,\ldots,a_m\}$ intersects
a generic point of the boundary of $\mathscr{P}_{\A}(K)$,
we can also get a feasible point of \reff{maxb:PAK}
by solving \reff{ck:f-in-QIk} (cf.~\cite{Nie-opcd}).
In the remaining cases, it is still an open question
to find a feasible point of \reff{maxb:PAK}
by using SOS relaxations, to the best of the author's knowledge.

\begin{exm}
We want to find $\lmd_1,\lmd_2$ such that
$
c - \lmd_1 a_1 - \lmd_2 a_2 \in P_{3,6},
$
where
\[
   c =  x_1^2( x_1^4 + x_2^2x_3^2 - x_1^2(x_2^2+x_3^2) ),
 a_1 =  x_2^2( x_2^4 + x_3^2x_1^2 - x_2^2(x_3^2+x_1^2) ),
 \]
 \[
 a_2 =  x_3^2( x_3^4 + x_1^2x_2^2 - x_3^2(x_1^2+x_2^2) ) .
\]
%
%
For $k=4$, \reff{ck:f-in-QIk} is feasible
with $(\lmd_1, \lmd_2) = (-1,-1)$.
\qed
\end{exm}

\subsection{Infeasibility certificates}

First, we prove a certificate for the infeasibility of \reff{opt:RAK}.
Suppose $a_1,\ldots,a_m \in \re[x]_{\A}$ and $b\in \re^m$ are given,
while $c$ is not necessarily.

\begin{lem} \label{lm:infea:RAK}
Let $K$ be as in \reff{def:K}.
Then, we have:
\bit

\item [(i)] The problem \reff{opt:RAK} is infeasible
if \reff{mincy:k-MOM} is infeasible for some order $k$;
\reff{mincy:k-MOM} is infeasible
if there exist $\lmd$ and $k$ such that
\be  \label{blmd<0:lmaPAK}
b^T\lmd < 0, \quad  \lmd_1 a_1 + \cdots + \lmd_m a_m \in Q_k(g) + I_{2k}(h).
\ee

\item [(ii)]  Suppose $I(h)+Q(g)$ is archimedean and
there exists $a \in \mbox{span}\{a_1,\ldots,a_m\}$ such that $a>0$ on $K$.
If \reff{opt:RAK} is infeasible, then
\reff{blmd<0:lmaPAK} holds for some $\lmd,k$
and \reff{mincy:k-MOM} is infeasible.

\eit

\end{lem}
\begin{proof}
(i) The problem \reff{mincy:k-MOM} is a relaxation of \reff{opt:RAK}, i.e.,
the set of feasible $y$ in \reff{opt:RAK}
is contained in that of \reff{mincy:k-MOM}. Clearly,
if \reff{mincy:k-MOM} is infeasible, then
\reff{opt:RAK} is also infeasible.

If \reff{blmd<0:lmaPAK} holds for some $\lmd,k$,
then \reff{mincy:k-MOM} must be infeasible, because
any feasible $y$ in \reff{mincy:k-MOM} would result in the contradiction
{\small
\[
0 > b^T\lmd = \sum_{i=1}^m \lmd_i \langle a_i, y \rangle
= \langle \sum_{i=1}^m \lmd_i a_i, y \rangle  \geq 0.
\]
}

(ii) Suppose \reff{opt:RAK} is infeasible. Consider the optimization problem
\be \label{0:ay=b:RAK}
\max \quad 0 \quad s.t. \quad
\langle a_i, y \rangle = b_i \, (i =1,\ldots, m), \quad
y \in \mathscr{R}_{\A}(K).
\ee
Its dual problem is
\be  \label{blmd:a-in-PAK}
\min_{ \lmd \in \re^m } \quad b^T\lmd \quad s.t. \quad
\lmd_1 a_1 + \cdots + \lmd_m a_m  \in \mathscr{P}_{\A}(K).
\ee
By the assumption, $\mathscr{R}_{\A}(K)$ and $\mathscr{P}_{\A}(K)$
are closed convex cones (cf. Proposition~\ref{lm:pro+dual}),
and \reff{blmd:a-in-PAK} has an interior point.
So, the strong duality holds and
\reff{blmd:a-in-PAK} must be unbounded from below
(cf. \cite[Theorem~2.4.I]{BTN}), i.e.,
there exists $\hat{\lmd}$ satisfying
\[
b^T \hat{\lmd} <0,  \quad
\hat{\lmd}_1 a_1 + \cdots + \hat{\lmd}_m a_m  \in \mathscr{P}_{\A}(K).
\]
By the assumption, there exists $\bar{\lmd}$ such that
$\bar{\lmd}_1 a_1 + \cdots + \bar{\lmd}_m a_m >0$ on $K$.
For $\eps >0$ small,
$\lmd:=\hat{\lmd} + \eps \bar{\lmd}$ satisfies \reff{blmd<0:lmaPAK}
for some $k$, by Theorem~\ref{thm:PutThm}.
By item (i), we know \reff{mincy:k-MOM} is infeasible.
\qed
\end{proof}

Here is an example for the infeasibility certificate \reff{blmd<0:lmaPAK}.

\begin{exm} \label{exm:5.4}
Let $\A =\N_6^2$ and $K=\mathbb{S}^1$ be the unit circle in $\re^2$,
defined by $h=(x_1^2+x_2^2-1)$ and $g=(0)$ as in \reff{def:K}.
We want to know whether there exists a measure $\mu$
supported in $\mathbb{S}^1$ such that
\[
\int  x_1^2 x_2^2 \mathtt{d} \mu = 1, \,
\int  (x_1^4+x_2^4) \mathtt{d} \mu = 1, \,
\int  (x_1^6+x_2^6) \mathtt{d} \mu = 1.
\]
This is equivalent to checking whether there exists
$y \in \mathscr{R}_{\A}(\mathbb{S}^1)$ satisfying
\[
\langle a_1, y \rangle = 1, \quad
\langle a_2, y \rangle = 1, \quad
\langle a_3, y \rangle = 1, \quad
\]
with $a_1 = x_1^2x_2^2$, $a_2 =x_1^4+x_2^4$, $a_3 = x_1^6+x_2^6$.
Indeed, such an $\A$-tms $y$ does not exist, because
\reff{blmd<0:lmaPAK} is satisfied for $\lmd=(-3,1,1)$:
$\lmd_1 + \lmd_2 + \lmd_3 <0$ and
\[
-3a_1+a_2+a_3 = 2(x_1^2-x_2^2)^2 + (x_1^4-x_1^2x_2^2+x_2^4) h
 \in I_6(h) + Q_3(g).
\]
By Lemma~\ref{lm:infea:RAK},
the above measure $\mu$ does not exist.
\qed
\end{exm}

\bigskip

Second, we give a certificate for the infeasibility of \reff{maxb:PAK}.
Suppose $c,a_1,\ldots,a_m \in \re[x]_{\A}$ are given,
while $b$ is not necessarily.

\begin{lem}  \label{crtf:PAK:infeas}
Let $K$ be compact and $c,a_1, \ldots, a_m \in \re[x]_{\A}$ be given.
\bit
\item [(i)] Problem \reff{maxb:PAK} is infeasible if there exists $y$ satisfying
\be \label{cy<0:ay=0}
c^Ty < 0, \quad
\langle a_i, y \rangle = 0 \, (i =1,\ldots, m), \quad y \in \mathscr{R}_{\A}(K).
\ee

\item  [(ii)] Suppose there does not exist $0 \ne a \in \mbox{span}\{a_1,\ldots,a_m\}$
such that $a\geq 0$ on $K$. If \reff{maxb:PAK} is infeasible,
then there exists $y$ satisfying \reff{cy<0:ay=0}.
\eit
\end{lem}

\begin{remark}
Clearly, if there exists $a \in \mbox{span}\{a_1,\ldots,a_m\}$
such that $a > 0$ on $K$, then \reff{maxb:PAK} must be feasible.
Therefore, for \reff{maxb:PAK} to be infeasible,
none of polynomials in $\mbox{span}\{a_1,\ldots,a_m\}$
can be positive on $K$.
So, the assumption in Lemma~\ref{crtf:PAK:infeas} (ii)
is almost necessary for \reff{maxb:PAK} to be infeasible.
Indeed, it cannot be removed. For instance,
consider $K = \mathbb{S}^1$ and $\A = \{ |\af| = 2\}$.
Choose $c,a_1$ such that
$
c(\lmd) = x_1x_2 - \lmd_1 x_1^2.
$
Clearly, $c(\lmd) \not\in \mathscr{P}_{\A}( \mathbb{S}^1 )$ for all $\lmd$.
For all $y \in \mathscr{R}_{\A}( \mathbb{S}^1 )$, if $\langle a_1, y \rangle =0$,
then $\langle c, y \rangle =0$. This is because
{\small
\[
\left| \int x_1 x_2 \mathtt{d} \mu  \right| \leq
\left( \int x_1^2 \mathtt{d} \mu \right)^{1/2}
\left( \int x_2^2 \mathtt{d} \mu \right)^{1/2}
\]
\noindent}for
all nonnegative measure $\mu$, by the Cauchy-Schwarz inequality.
So, there is no $y$ satisfying \reff{cy<0:ay=0},
while \reff{maxb:PAK} is infeasible.
\end{remark}

\begin{proof}[of Lemma~\ref{crtf:PAK:infeas}]\,
(i) Suppose \reff{cy<0:ay=0} holds.
If \reff{maxb:PAK} has a feasible $\lmd$, then we get
\[
0 \leq \langle c(\lmd), y \rangle = \langle c, y \rangle
- \lmd_1 \langle a_1, y \rangle - \cdots - \lmd_m \langle a_m, y \rangle
= \langle c, y \rangle < 0,
\]
a contradiction. So \reff{maxb:PAK} must be infeasible
if \reff{cy<0:ay=0} is satisfied.

(ii) Without loss of generality, we can assume $a_1,\ldots,a_m$
are linearly independent in the quotient space $\re[x]/I(K)$
(i.e., the space of polynomial functions defined on $K$ \cite{CLO}, where
$I(K)$ is the ideal of all polynomials $p$ such that $p \equiv 0$ on $K$).
We show that there exists $T>0$ such that
\be \label{c+a:in:B(0,T)}
\mathscr{P}_{\A}(K) \cap (c+\mbox{span}\{a_1,\ldots,a_m\}) \subseteq B(0,T).
\ee
(The left above intersection might be empty.)
Suppose otherwise such $T$ does not exist, then there exists a sequence $\{\lmd^k\}$
such that $\|\lmd^k\|_2 \to \infty$ and
$c(\lmd^k) \in \mathscr{P}_{\A}(K)$ for all $k$.
The sequence $\{ \lmd^k/\|\lmd^k\|_2 \}$ is bounded.
We can generally assume $\lmd^k/\|\lmd^k\|_2 \to \lmd^* \ne 0$.
Clearly, $c(\lmd^k)/\|\lmd^k\|_2 \in \mathscr{P}_{\A}(K)$ for all $k$. So,
\[
c(\lmd^k)/\|\lmd^k\|_2 \to a^*:=
-( \lmd_1^*a_1 + \cdots + \lmd_m^*a_m )  \in \mathscr{P}_{\A}(K).
\]
Since $a_1,\ldots,a_m$ are linearly independent in $\re[x]/I(K)$ and $\lmd^*\ne 0$,
we know $a^*|_K \not\equiv 0$ and $a^*|_K \geq 0$.
This contradicts the given assumption.
So \reff{c+a:in:B(0,T)} must be satisfied for some $T>0$.
Let
\[
\mathscr{C}_1 = \{ p \in  \mathscr{P}_{\A}(K) \mid \|p\|_2 \leq T\},  \quad
\mathscr{C}_2 = c+ span\{a_1, \ldots, a_m\}.
\]
By \reff{c+a:in:B(0,T)}, \reff{maxb:PAK} is infeasible
if and only if $\mathscr{C}_1 \cap \mathscr{C}_2 = \emptyset$.
Because $K$ is compact, the set $\mathscr{C}_1$ is compact convex, and
$\mathscr{C}_2$ is closed convex.
By the strict convex set separation theorem, they do not intersect if and only if
there exists $y \in \re^{\A}$ and $\tau \in \re$ such that
\[
\langle p, y \rangle > \tau\quad \forall \, p \in \mathscr{C}_1,
\]
\[
\langle p, y \rangle < \tau \quad \forall \, p \in \mathscr{C}_2.
\]
The first above inequality implies $\tau < 0$
and $y \in \mathscr{R}_{\A}(K)$, and the second one
implies $c^Ty < 0$ and $\langle a_i, y \rangle = 0$ for all $i$.
Thus, this $y$ satisfies \reff{cy<0:ay=0}.
\qed
\end{proof}

The certificate \reff{cy<0:ay=0} can be checked by solving
the feasibility problem:
\be \label{cy=-1:ay=0}
c^Ty = -1, \quad
\langle a_i, y \rangle = 0 \, (i =1,\ldots, m), \quad y \in \mathscr{R}_{\A}(K).
\ee

\begin{exm}
Let $K=\mathbb{S}^2$ and $\A = \{\af \in \N^n: |\af| = 6\}$.
Then $\mathscr{P}_{\A}(K)$ equals $P_{3,6}$,
the cone of nonnegative ternary sextic forms.
We want to know whether there exist
$\lmd_1, \lmd_2, \lmd_3$ such that
{\small
\[
\underbrace{x_1^2x_2^2(x_1^2+x_2^2-4x_3^2) +x_3^6}_{c}
- \lmd_1 \underbrace{x_1^3x_2^3}_{a_1}
- \lmd_2 \underbrace{x_1^3x_3^3}_{a_2}
- \lmd_3 \underbrace{x_2^3x_3^3}_{a_3}
\in P_{3,6}.
\]
}\noindent
Indeed, there are no $\lmd_1, \lmd_2, \lmd_3$ satisfying the above.
To get a certificate for this, solve the feasibility problem \reff{cy=-1:ay=0}.
It has a feasible tms $y$ that admits the finitely atomic measure
$
\frac{27}{4} \big( \dt_{(1, 1, 1)/\sqrt{3}} + \dt_{(-1, 1, 1)/\sqrt{3}}
+ \dt_{(1, -1, 1)/\sqrt{3}}  + \dt_{(1, 1, -1)/\sqrt{3}} \big).
$
\qed
\end{exm}

%
%

\bigskip
\leftline{\bf Appendix: Checking $K$-fullness}
\bigskip

\noindent
Recall that $\re[x]_{\A}$ is $K$-full if there exists
$p \in \re[x]_{\A}$ that is positive on $K$.
We can check whether $\re[x]_{\A}$ is $K$-full or not as follows.
Clearly, $\re[x]_{\A}$ is $K$-full if and only if
there exists $\lmd \in \re^{\A}$ such that
{\small
\be \label{chec:K-full}
\sum_{\af \in \A} \lmd_\af x^\af - 1 \in \mathscr{P}_{\A^\prm}(K),
\ee
} \noindent
where $\A^\prm = \A \cup \{ 0 \}$.
Since $ 1\in \mathscr{P}_{\A^\prm}(K)$, $\re[x]_{\A^\prm}$ is always $K$-full.
Thus, checking $K$-fullness is reduced to solving a feasibility/infeasiblity issue.
This will be discussed in Section~\ref{sec:feas}.
Suppose $K$ is a compact semialgebraic set as in \reff{def:K}.
If $\re[x]_{\A}$ is $K$-full, we can get a
$\lmd$ satisfying \reff{chec:K-full} (cf.~Section~5.1).
If $\re[x]_{\A}$ is not $K$-full,
we can get a certificate for nonexistence of such $\lmd$,
under a general assumption (cf.~Lemma~\ref{crtf:PAK:infeas}).

\end{document}